\documentclass[11pt,reqno]{amsart}
\usepackage{amsmath, amssymb, amsthm}
\usepackage{tensor}
\usepackage{enumerate}
\usepackage[colorlinks=true,allcolors=blue]{hyperref}
\usepackage{blkarray}
\usepackage[numbers]{natbib}
\usepackage[font=footnotesize]{caption}

\usepackage{tikz}
\usetikzlibrary{cd,arrows,calc}

\newcommand{\C}{\mathbb{C}}
\newcommand{\R}{\mathbb{R}}
\newcommand{\Z}{\mathbb{Z}}
\newcommand{\N}{\mathbb{N}}
\newcommand{\Q}{\mathbb{Q}}

\newcommand{\bK}{\mathbb{K}}
\newcommand{\bT}{\mathbb{T}}

\newcommand{\cR}{\mathcal{R}}

\newcommand{\Cuntz}[1]{\mathcal{O}_{#1}}
\newcommand{\Cliff}[1]{\C\ell(#1)}
\newcommand{\Cln}[1]{\C\ell_{#1}}

\newcommand{\Aut}[1]{\text{\normalfont Aut}(#1)}

\newcommand{\AutSt}[2]{\operatorname{Aut}_{#1}(#2)}
\newcommand{\eqAut}[2]{\operatorname{Aut}_{#1}(#2)}
\newcommand{\eqAutId}[2]{\operatorname{Aut}_{#1,0}(#2)}
\newcommand{\Autf}[3]{\operatorname{Aut}_{#1}^{#2}(#3)}

\newcommand{\Endo}[1]{\operatorname{End}\left(#1\right)}

\newcommand{\Proj}[2]{\operatorname{Proj}_{#1}(#2)}

\newcommand{\id}[1]{\operatorname{id}_{#1}}
\newcommand{\obj}[1]{\text{obj}(#1)}

\newcommand{\cA}{\mathcal{A}}
\newcommand{\cP}{\mathcal{P}}

\newcommand{\cD}{\mathcal{D}}
\newcommand{\cG}{\mathcal{G}}

\newcommand{\cI}{\mathcal{I}}

\newcommand{\cO}{\mathcal{O}}

\newcommand{\cEG}[1]{\mathcal{E}_{#1}}
\newcommand{\cBG}[1]{\mathcal{B}_{#1}}

\newcommand{\op}{\text{op}}
\newcommand{\End}{\text{End}}

\DeclareMathOperator{\Ad}{Ad}

\DeclareMathOperator{\colim}{colim}
\DeclareMathOperator{\hocolim}{hocolim}

\newcommand{\grS}{\widehat{\mathcal{S}}}

\newcommand{\Top}{\mathcal{T}op}
\newcommand{\GTop}{G\mathcal{T}op}
\newcommand{\GCStar}{G\text{-}C^*\textbf{Alg}}

\newtheorem{theorem}{Theorem}[section]
\newtheorem{lemma}[theorem]{Lemma}
\newtheorem{corollary}[theorem]{Corollary}
\newtheorem{prop}[theorem]{Proposition}

\theoremstyle{definition}
\newtheorem{definition}[theorem]{Definition}
\newtheorem{remark}[theorem]{Remark}

\makeatletter
\newcommand{\extp}{\@ifnextchar^\@extp{\@extp^{\,}}}
\def\@extp^#1{\mathop{\bigwedge\nolimits^{\!#1}}}
\makeatother

\newcommand{\Fam}{\mathcal{F}}
\newcommand{\EFam}[2]{E_{#1}#2}

\begin{document}
\title[Units of $\Z/n\Z$-equivariant $K$-theory and bundles of UHF-algebras]{Units of $\Z/n\Z$-equivariant $K$-theory\\and bundles of UHF-algebras}
\author{Valerio Bianchi \and Ulrich Pennig}
\address{Cardiff University, School of Mathematics, Senghennydd Road, Cardiff, CF24 4AG, Wales, UK}
\email{BianchiV1@cardiff.ac.uk}
\email{PennigU@cardiff.ac.uk}

\begin{abstract}
We consider infinite tensor product actions of a finite cyclic group $G$ on the UHF-al\-ge\-bra $D = \Endo{V}^{\otimes \infty}$ for a finite-dimensional projectively faithful unitary $G$-repre\-sen\-ta\-tion $V$ and determine the equivariant homotopy type of the group $\Aut{D \otimes \bK}$, where~$\bK$ are the compact operators on $\ell^2(G) \otimes H_0$ for a separable Hilbert space $H_0$ with $\dim(H_0) = \infty$. We show that this group carries an equivariant infinite loop space structure revealing it as the first space of a naive $G$-spectrum. We identify this spectrum with the positive units $gl_1(KU^D)_+$ of equivariant $KU^D$-theory, which was not possible in the case of circle actions considered by Evans and the second author. Here, $KU^D$ is a $G$-spectrum representing $X \mapsto K_*^G(C(X) \otimes D)$. As a consequence the first group of the cohomology theory associated to $gl_1(KU^D)_+$ classifies equivariant $D \otimes \bK$-bundles over finite $G$-CW-complexes. 
\end{abstract}

\maketitle

\section{Introduction}

Let $R$ be a commutative ring. Forgetting its additive structure we may focus on the multiplication in $R$ and consider its group of units $GL_1(R)$. The analogue of commutative rings in stable homotopy theory are commutative ring spectra. Famous examples here include the ones representing ordinary cohomology or topological $K$-theory. Similarly to the algebraic setting an associative ring spectrum $E$ has a space of units $GL_1(E)$, which turns out to be the zeroth space of a spectrum $gl_1(E)$ of units in case $E$ is an $E_\infty$-ring spectrum. This theory of units originated in Sullivan's MIT notes \cite{paper:Sullivan}, where they were used to study obstructions to the orientability of vector bundles and spherical fibrations. 

Each model of ring spectra has its own definition of the corresponding units. The definition for $E_\infty$-ring spectra goes back to May, Quinn, Ray and Tornehave \cite{book:May}. The units of a commutative symmetric ring spectrum were initially defined in work by Schlichtkrull \cite{paper:Schlichtkrull}. It was shown by Lind in~\cite{paper:Lind} that the comparison functors between the various models of ring spectra give rise to weak equivalences between their respective units. 

The motivation for the current paper was an observation that interlinks the unit spectrum $gl_1(KU^\cD)$ of a variant of complex topological $K$-theory  with the theory of $C^*$-algebras. Here, $\cD$ denotes a strongly self-absorbing $C^*$-algebra. A unital $C^*$-algebra $\cD$ belongs to this class if there is an isomorphism $\varphi \colon \cD \to \cD \otimes \cD$ that is asymptotically unitarily equivalent to the left tensor embedding (see Definition~\ref{ssa}). These algebras play a cornerstone role in the classification programme of separable nuclear simple $C^*$-algebras. Examples include infinite tensor products of matrix algebras (i.e.\ infinite UHF-algebras), the Jiang-Su algebra $\mathcal{Z}$ and the infinite Cuntz algebra $\Cuntz{\infty}$. To each such algebra $\cD$ one can associate a commutative symmetric ring spectrum $KU^\cD_*$ that represents the cohomology theory 
\[
    X \mapsto K_*(C(X) \otimes \cD) 
\]
on finite CW-complexes \cite[Sec.~4.1]{paper:DP-Units}, i.e.\ the operator-algebraic $K$-theory of the $C^*$-algebra of continuous $\cD$-valued functions on $X$. In particular, we have weak equivalences $KU^\C \simeq KU^\mathcal{Z} \simeq KU^{\Cuntz{\infty}}$ induced by the unit maps $\C \to \mathcal{Z}$ and $\C \to \Cuntz{\infty}$, and these three spectra represent complex topological $K$-theory $X \mapsto K^*(X)$ by the Serre-Swan theorem. 

Let $\bK$ denote the $C^*$-algebra of compact operators on a separable infinite-dimensional Hilbert space. Together with Dadarlat the second author constructed a map $\Aut{\cD \otimes \bK} \to \Omega KU^\cD_1$ in \cite{paper:DP-Units}, which factors through 
\begin{equation} \label{eqn:noneq_Aut_GL1_Map}
    \Aut{\cD \otimes \bK} \to GL_1(KU^\cD)
\end{equation} 
and extends to a map of infinite loop spaces. The abelian group $K_0(\cD)$ has a natural order structure and the spectrum $gl_1(KU^\cD)$ restricts to the spectrum of positive units $gl_1(KU^\cD)_+$ by pullback via the group homomorphism $GL_1(K_0(\cD))_+ \to GL_1(K_0(\cD))$. For $\cD\neq\C$ the map \eqref{eqn:noneq_Aut_GL1_Map} gives a weak equivalence of infinite loop spaces $\Aut{\cD \otimes \bK} \simeq \Omega^\infty gl_1(KU^\cD)_+$. In case $\cD$ is purely infinite the order structure on $K_0(\cD)$ is trivial (i.e.\ $K_0(\cD)_+ = K_0(\cD)$) and we have $\Aut{\cD \otimes \bK} \simeq \Omega^\infty gl_1(KU^\cD)$ \cite[Thm.~1.1]{paper:DP-Units}. 

Important consequences of this result arise for twisted $K$-theory, which was initially developed in \cite{paper:DonovanKaroubi}, see \cite{paper:KaroubiSurvey} for a survey. The infinite loop map from the last paragraph gives rise to an equivalence
\[
    B\Aut{\cD \otimes \bK} \simeq BGL_1(KU^\cD)_+\ .
\] 
Since the right hand side is the classifying space of (positive) twists of~$KU^\cD$, we obtain a description of twisted $KU^\cD$-theory in terms of bundles of strongly self-absorbing $C^*$-algebras. As a special case it includes the geometric twists, i.e.\ those that are classified by $K(\Z,3) \simeq B\Aut{\bK}$, or equivalently the group $[X, K(\Z,3)] \cong H^3(X,\Z)$, considered in \cite{paper:AtiyahSegal, paper:Rosenberg}. 

Equivariant twisted $K$-theory (with a twist given by an equivariant bundle of compact operators) has found applications in mathematical physics through a deep theorem by Freed, Hopkins and Teleman \cite{paper:FHT3}. This motivates the search for operator-algebraic models of equivariant twists beyond those arising from bundles of compact operators, which in turn requires the construction of a $G$-space $BGL_1(KU)$.
At the point of writing the only results in the literature covering units of equivariant ring spectra are to the best of our knowledge \cite{paper:Santhanam} and \cite[Ex.~5.1.17]{book:SchwedeGlobal}. Nevertheless, candidates for equivariant higher twists over $SU(n)$ in terms of $C^*$-algebra bundles have been constructed in \cite{paper:EvansPennigTwists}. 

In this paper we present an important first advance towards an equivariant refinement of the higher Dixmier-Douady theory developed in~\cite{paper:DP-Units}. Our approach combines methods from equivariant stable homotopy theory with topological considerations about automorphism groups of $C^*$-algebras. As a starting point we picked a setting that is very explicit: we fix the group $G$ to be $\Z/n\Z$ for $n \in \N$, and we consider a UHF-algebra~$D$ given by the infinite tensor product of a unitary finite-dimensional $G$-representation $V$, i.e.
\[
    D = \Endo{V}^{\otimes \infty}\ .
\]
This is a $G$-$C^*$-algebra in a natural way when equipped with the infinite tensor product of the action by conjugation. Some of the results obtained in~\cite{paper:EP-Circle} for circle actions on UHF-algebras carry over immediately to the case of $\Z/n\Z$-actions. 
%The restriction to finite cyclic groups has two benefits: There are only two subgroups to consider, and 
We then determine the $G$-equivariant homotopy type of $\Aut{D \otimes \bK}$. In particular, we show that the automorphism group is an equivariant infinite loop space associated to a naive $G$-spectrum, by which we mean a spectrum indexed by the trivial $G$-universe. Unlike the case of circle actions considered in \cite{paper:EP-Circle} we are able to identify this spectrum in our main result:
\begin{theorem}[Thm.~\ref{thm:main_thm}] \label{thm:intro_main}
    Let $G = \Z/n\Z$ for $n \in \N$. Let $V$ be a finite-dimensional projectively faithful unitary $G$-representation and let
    \[
        D = \Endo{V}^{\otimes \infty}
    \]
    be the associated $G$-$C^*$-algebra. The group $\Aut{D \otimes \bK}$ is a $G$-equivariant infinite loop space underlying a naive $G$-spectrum $aut(D \otimes \bK)$ and we have an equivalence of naive $G$-spectra
    \[
        aut(D \otimes \bK) \simeq gl_1(KU^D)_+\ .
    \]
\end{theorem}

The group $\Aut{D \otimes \bK}$ has two equivariant deloopings: one arising from the group structure, the other one from the tensor product. Just as in the non-equivariant setting these two agree. As a result the first group of the cohomology theory associated to $gl_1(KU^D)_+$ classifies equivariant $C^*$-algebra bundles with fibre $D \otimes \bK$ in the following sense:
\begin{corollary}[Cor.~\ref{cor:bdl_classification}]
Let $D$ be as in Theorem~\ref{thm:intro_main} and let $E_D^*(X) = gl_1(KU^D)_+^*(X)$ be the $\Z/n\Z$-equivariant cohomology theory associated to the spectrum $gl_1(KU^D)_+$. Let~$X$ be a finite $\Z/n\Z$-CW-complex. Then 
\[
    E_D^0(X) \cong [X, \Aut{D \otimes \bK}]_{\Z/n\Z} \quad \text{and} \quad E_D^1(X) \cong [X, B\Aut{D \otimes \bK}]_{\Z/n\Z}\ .
\]
In particular, isomorphism classes of locally trivial $(\Z/n\Z, D \otimes \bK)$-bundles as in Definition \ref{def:eq_bundles} over $X$ form a group with respect to the fibrewise tensor product that is isomorphic to $E^1_D(X)$.
\end{corollary}

The article is structured as follows. In Section \ref{prel_1} we describe an equivariant infinite loop space machine for finite groups based on diagram spaces, called commutative $\cI$-$G$-monoids, whose output is a naive positive $\Omega$-$G$-spectrum. In Section \ref{prel_2} we turn to operator algebras: we provide some background on strongly self-absorbing $C^*$-algebras and describe our setup.\\
In Section \ref{sec3} we construct an equivariant spectrum of units for a localisation of equivariant $K$-theory. This is done in two steps: in Section~\ref{spectrum} we refine the construction of a commutative symmetric ring spectrum $KU^D$ associated to a strongly self-absorbing $C^*$-algebra $D$ to the equivariant case (Definition~\ref{def:KUD}). Then we construct its equivariant units in Section~\ref{equnits} and lift them to a naive $G$-spectrum (Definition~\ref{def:spec_units}). \\
Our goal is to compare this $G$-spectrum of units to the one associated to the automorphism group of the stabilisation of $D=\Endo{V}^{\otimes \infty}$ for an action of $G=\Z/n\Z$. To do that, we interpret $\Aut{D \otimes \bK}$ as the first space in a commutative $\cI$-$G$-monoid $\cG_D$ in Section \ref{G_D}. In Section~\ref{Aut} we then collect a series of results which provide the equivariant homotopy type of this group, mainly using techniques from \cite{paper:EP-Circle}. We also compute the homotopy groups of the equivariant automorphisms in Corollary~\ref{cor:pikAut}.\\
The main result of this paper is contained in Section \ref{final_section}, where we explicitly describe a map between the $\Z/n\Z$-spectrum coming from the equivariant infinite loop space structure of $\Aut{D \otimes \bK}$ and the spectrum of $\Z/n\Z$-equivariant units of $KU^D$, and prove that it is an equivalence onto the positive units (Theorem~\ref{thm:main_thm}). In particular, this provides a classification of $\Z/n\Z$-equivariant $\Aut{D \otimes \bK}$-bundles in terms of an equivariant cohomology theory (Corollary~\ref{cor:bdl_classification}). We discuss a model for the classifying space of such equivariantly locally trivial bundles in Section \ref{final_1}, and we state and prove the main theorems in Section \ref{final_2}.\\
Constructing the equivariant classifying space $B\Aut{D \otimes \bK}$ as the geometric realisation of the nerve only works, if $\Aut{D \otimes \bK}$ is equivariantly well-pointed. Proving this is the content of Lemma~\ref{lem:G-cofib} in the appendix. There we also prove a few other topological results that might be of independent interest: for example, we show that $\Aut{D \otimes \bK}$ has the $G$-homotopy type of a $G$-CW-complex (Theorem~\ref{lem:GCW-complex}) and that an evaluation map $\Aut{A} \to \Proj{}{A}$ for a finite group $G$ and a $G$-$C^*$-algebra $A$ is a Hurewicz $G$-fibration when suitably restricted in the codomain (Lemma~\ref{lem:Hurewicz}). 

\subsection*{Acknowledgements} 
The authors would like to thank John A.\ Lind, Steffen Sagave and Stefan Schwede for helpful discussions. Some parts of this work emerged from the PhD thesis \cite{paper:Thesis} of the first author, which was funded by the EPSRC grant EP/W52380X/1 (project reference: 2601068).

\section{Preliminaries}

\subsection{Preliminaries from equivariant homotopy theory}
\label{prel_1}

Let $G$ be a finite group and let $\GTop$ be the category of compactly generated weak Hausdorff $G$-spaces and $G$-equivariant maps. Note that $\GTop$ has a model structure whose weak equivalences are morphisms $f:X \rightarrow Y$ such that $f^H: X^H \rightarrow Y^H$ is a weak equivalence in the underlying category $\Top$, for all subgroups $H \le G$. For $X, Y \in \obj{\GTop}$ we write $[X,Y]_G$ for the $G$-equivariant homotopy classes of $G$-maps; if $X$ and $Y$ have $G$-invariant basepoints, we denote by $[X,Y]_{\ast,G}$ the $G$-equivariant pointed homotopy classes of pointed $G$-maps.

The orbits $G/H$ form a full subcategory of $\GTop$ which we denote by~$\cO_G$, the orbit category. Note that for subgroups $K \le H \le G$, we have a natural map of $G$-sets $G/K \rightarrow G/H$, and a corresponding inclusion map of fixed points $X^H \rightarrow X^K$ in the opposite direction for a $G$-space $X$. This means we have a functor
\[\cO_G^{\op} \rightarrow \Top \qquad \text{ given by } \qquad G/H \mapsto X^H\ .\]
\begin{definition}
    An $\cO_G$-\textit{space} is a functor $\cO_G^\op \rightarrow \Top$. Denote by \emph{$\cO_G$-spaces} the category of $\cO_G$-spaces and continuous natural transformations.
\end{definition}
As described above there is a fixed point functor $\GTop \rightarrow \cO_G$-spaces defined on objects by
\[
X \mapsto \big((G/H) \mapsto X^H \big)\ .
\]
By Elmendorf's theorem \cite{paper:Elmendorf} this functor is an equivalence on homotopy categories. Its homotopy inverse can be constructed explicitly as the geometric realization of the simplicial bar construction:  
\begin{equation}\label{eqn:elmendorf}
\Phi \colon \cO_G\text{-spaces} \to \GTop \quad, \quad X \mapsto |B(X,\cO_G,M)|\ ,
\end{equation}
where $M$ is the forgetful functor realising orbits as spaces (see \cite[Sec. 8.8]{book:Kervaire} for details). For pointed $\cO_G$-spaces we use the pointed functor $\Phi_*$, defined by the reduced bar construction; in particular, $\Phi_*(*)=*$.

Following \cite{paper:Lind} we now define two diagram categories of $G$-spaces. Let $\mathcal{I}$ be the category of finite sets $\textbf{n}=\{1,\dots,n\}$ (including $\mathbf{0} = \emptyset$) and injective functions. It is a symmetric monoidal category via concatenation of finite ordered sets $\textbf{m} \sqcup \textbf{n}$ along with symmetric structure maps $\textbf{m} \sqcup \textbf{n} \rightarrow \textbf{n} \sqcup \textbf{m}$ given by the obvious shuffles.

\begin{definition}
An $\mathcal{I}$-$G$-\textit{space} is a functor $X \colon \mathcal{I} \rightarrow \GTop$.
An $\mathcal{I}$-$G$-space~$X$ is called a (commutative) $\mathcal{I}$-$G$-\textit{monoid} if it comes equipped with a natural transformation $\mu: X \times X \rightarrow X \circ \sqcup$ which is associative, unital (and commutative) in the sense of \cite[Def.~2.2]{paper:DP-Units}.
\end{definition}

Let $\Gamma^{\op}$ be the category of finite based sets $\textbf{n}^+=\{0,1,\dots,n\}$ with basepoint~$0$ and based maps.

\begin{definition}
A $\Gamma$-$G$-\textit{space} is a functor $A:\Gamma^{\op} \rightarrow \GTop$ such that $A(\mathbf{0}^+)=*$. A $\Gamma$-$G$-space is called \textit{special} if for every $n$ the $G$-map 
\[A(\mathbf{n}^+) \rightarrow \prod_{i=1}^{n}A(\textbf{1}^+)\]
induced by projections is a weak equivalence. We refer to this map as the equivariant Segal map. 
\end{definition}

Note that the maps $A(\mathbf{0}^+) \to A(\mathbf{n}^+)$ induced by $\mathbf{0}^+ \to \mathbf{n}^+$ give each $A(\mathbf{n}^+)$ a basepoint. 
We call a $\Gamma$-$G$-space $A$ \emph{cofibrant} if $A(\varphi)$ is a $G$-cofibration for every injective based map $\varphi \colon \mathbf{m}^+ \to \mathbf{n}^+$. In particular, each $A(\mathbf{n}^+)$ is then nondegenerately based.
%We call a $\Gamma$-$G$-space $A$ equivariantly well-pointed if the basepoint inclusion $*\to A(\mathbf n^+)$ is a $G$-cofibration for every $n$.

The notions of $\cI$-$\cO_G$-\textit{spaces} and (special) $\Gamma$-$\cO_G$-\textit{spaces} are defined analogously by replacing the target category by $\cO_G$-spaces.

To any commutative $\mathcal{I}$-$G$-monoid one can associate a special $\Gamma$-$G$-space. This works as in the non-equivariant case \cite[Construction~12.1]{paper:Lind}. We sketch the argument here for the reader's convenience and refer to \cite{paper:Thesis} for the details.

Let $X$ be a commutative $\cI$-$G$-monoid. Let $P(\textbf{n}^+)$ be the poset of subsets $0 \notin A \subset \textbf{n}^+$ and define $\mathcal{I}(\textbf{n}^+)$ to be the category of functors
$\theta: P(\textbf{n}^+) \rightarrow \mathcal{I}$ preserving disjoint unions.

%A functor $\theta \in \mathcal{I}(\textbf{n}^+)$ is given by the data of a collection of objects $\theta_{i}:=\theta(\{i\}) \in \mathcal{I}$ for each $1 \le i \le n$ (by definition, $\theta(\emptyset)=\textbf{0}$), plus morphisms $\theta(i) \rightarrow \theta(A)$ for each $A \subset \{1,\dots,n\}$ which assemble into a canonical isomorphism 
%\begin{equation} \label{eqn:theta}
 %   \bigoplus_{i \in A} \theta_{i} \cong \theta(A)\ .
%\end{equation}
Let $\theta_{i}:=\theta(\{i\}) \in \mathcal{I}$ for each $1 \le i \le n$ (by definition, $\theta(\emptyset)=\textbf{0}$)  and consider the forgetful functor $u: \mathcal{I}(\textbf{n}^+) \rightarrow \mathcal{I}^{n}$ that sends $\theta$ to $(\theta_{1}, \dots, \theta_{n})$. We get a functor $X(\textbf{n}^+):\mathcal{I}(\textbf{n}^+) \rightarrow \GTop$ given by the composite
\[X(\textbf{n}^+): \mathcal{I}(\textbf{n}^+) \xrightarrow{u} \mathcal{I}^{n} \xrightarrow{X^{n}} \GTop\]
\[\theta \mapsto \prod_{i=1}^{n}X(\theta_{i})\ .\]
Finally, define a functor $H_{\mathcal{I}}X: \Gamma^{\op} \rightarrow \GTop$ as the homotopy colimit (in the sense of \cite[Def. 5.8.1]{book:Kervaire})
\[H_{\mathcal{I}}X(\textbf{n}^+)=\hocolim_{\mathcal{I}(\textbf{n}^+)}X(\textbf{n}^+)\]
and set $H_{\mathcal{I}}X(\textbf{0}^+)=*$. Note that $H_{\mathcal{I}}X(\textbf{1}^+)=\hocolim_{\mathcal{I}}X=:X_{h \cI}$.
One can check that the commutative monoid structure of $X$ ensures that this definition is functorial \cite{paper:Thesis}.

Recall that a simplicial $G$-space $X_\bullet$ is called \emph{Reedy cofibrant} if for every $n \geq 1$ the inclusion of the latching space $L_nX = \bigcup_{i=0}^{n-1} s_i(X_{n-1})$ into $X_n$ is a $G$-cofibration \cite[Def.~1.10]{paper:MayMerlingOsorno}. This holds in particular if all degeneracy maps of $X_\bullet$ are $G$-cofibrations \cite[Lem.~1.11]{paper:MayMerlingOsorno}. The geometric realisation of Reedy cofibrant simplicial $G$-spaces is homotopically well-behaved: it maps levelwise weak $G$-equivalences to weak $G$-equivalences and levelwise $G$-cofibrations to $G$-cofibrations \cite[Thm.~1.12 and Thm.~1.13]{paper:MayMerlingOsorno}.

\begin{remark} \label{rem:Bousfield-Kan}
We will use the Bousfield-Kan model for the homotopy colimit, which commutes with taking fixed points: for a diagram $X \colon \mathcal{C} \to \GTop$ we have
\[
    \hocolim_{\mathcal{C}} X = |X_\bullet| \ , \quad \text{where} \quad X_k = \coprod_{c_0 \to \dots \to c_k} X(c_0)
\]
is the simplicial replacement. Since $G$ acts on the values of $X$ but not on the indexing category, each summand of $X_k$ is $G$-invariant, so 
\[
    X_k^H = \coprod_{c_0 \to \dots \to c_k} X(c_0)^H \ ,
\]
and $X^H_\bullet$ is the simplicial replacement of the diagram~$X^H$. The degeneracy maps of $X_\bullet$ are inclusions of summands, hence $G$-cofibrations, so $X_\bullet$ is Reedy cofibrant \cite[Lem.~1.11]{paper:MayMerlingOsorno}. Therefore $|X_\bullet|^H \cong |X^H_\bullet|$ by \cite[Lem.~12.3]{paper:MayMerlingOsorno}, and thus $(\hocolim_{\mathcal{C}} X)^H \cong \hocolim_{\mathcal{C}} (X^H)$.
\end{remark}

\begin{prop}
The functor $H_{\cI}X$ is a special $\Gamma$-$G$-space.
\end{prop}

\begin{proof}
We need to check that the equivariant Segal map
\[H_{\cI}X(\mathbf{n}^+) \rightarrow \big(\prod_{i=1}^{n}H_{\cI}X(\textbf{1}^+)\big)\]
is a weak equivalence. Using the canonical isomorphism
\[(\hocolim_{\mathcal{D}}X)^{n} \cong \hocolim_{\mathcal{D}^{n}}X^{n}\]
this reduces to checking that the map of homotopy colimits induced by the forgetful functor $u$
\[ \hocolim_{\cI(\textbf{n}^+)}X(\textbf{n}^+) \rightarrow \hocolim_{\cI^{n}}X^{n}\]
is a weak equivalence, which in turn follows from the Bousfield-Kan cofinality criterion \cite[Thm. 5.8.15]{book:Kervaire} applied to $u: \cI(\textbf{n}^+) \rightarrow \cI^{n}$, provided we can show that $u$ is homotopy final in the sense of \cite[Def.~5.8.13]{book:Kervaire}. For each $k=(k_1,\dots,k_n) \in \cI^n$, let $\theta:P(\textbf{n}^+) \rightarrow \cI$ be the functor defined by
\[
\theta(A)=\bigoplus_{i \in A} k_i\ .
\]
Then
\[\bigl(k, \theta, \text{id}: k \rightarrow u(\theta)\bigr) \in (k/u)\]
is an initial object for the comma category $(k/u)$. Thus $(k/u)$ has contractible nerve for every $k$, and $u$ is
homotopy final.
\qedhere
 
\end{proof}

\begin{remark}
    In the rest of the paper we will sometimes work with $\cO_G$-spa\-ces instead of $G$-spaces. The $\Gamma$-space construction can be repeated verbatim when $X$ is a commutative $\cI$-$\cO_G$-monoid. The output in this case is a special $\Gamma$-$\cO_G$-space, which we can turn into a special $\Gamma$-$G$-space again by applying the pointed functor $\Phi_*$ levelwise. Hence, we will consider $H_{\cI}$ as a functor from commutative $\cI$-$\cO_G$-monoids to $\Gamma$-$G$-spaces. Observe that the functor $\GTop \to \cO_G\text{-spaces}$ preserves products on the nose. Hence, if we denote the category of commutative monoids by $\text{Com}(-)$, then the diagram
    \[
        \begin{tikzcd}
            \text{Com}(\GTop^\cI) \ar[r] \ar[d] & \text{Com}(\cO_G\text{-spaces}^{\cI}) \ar[d] \\
            \text{Fun}(\Gamma^\op, \GTop)\ar[r] & \text{Fun}(\Gamma^\op, \cO_G\text{-spaces})
        \end{tikzcd}
    \]
    commutes. Note that the lower horizontal arrow is a levelwise equivalence by Elmendorf's theorem and Remark~\ref{rem:Bousfield-Kan}.
%    Here a key observation is that the functor $\GTop \rightarrow \cO_G$-spaces preserves products on the nose, hence allowing us to compare the output of this machinery to the one that we get by starting with the corresponding commutative $\cI$-$G$-monoid. As one would expect the two resulting objects are homotopy equivalent, because of Elmendorf's theorem and the fact that the homotopy colimit construction commutes with taking fixed points.
\end{remark}
    
%Before applying the Segal machine, we replace each $\Gamma$-$G$-space functorially by the cofibrant model supplied by \cite[Prop. 2.12]{paper:MayMerlingOsorno}.
We recall the following equivariant generalization \cite[Def. 2.23] {paper:MayMerlingOsorno} of the simplicial Segal machine, whose output is a naive $G$-spectrum.

\begin{definition}
\label{def_E}
    Let $A$ be a cofibrant special $\Gamma$-$G$-space. For each $p$, let $A[\textbf{p}^+]$ be the $\Gamma$-$G$-space that sends $\textbf{q}^+$ to $A(\textbf{p}^+ \wedge \textbf{q}^+)$. Define the classifying $\Gamma$-$G$-space $\mathbb{B}A$ to be the $\Gamma$-$G$-space whose $p$-th $G$-space is the geometric realization $|A[\textbf{p}^+]|$ (defined by precomposition with the functor $\Delta^{\text{op}} \rightarrow \Gamma^{\text{op}}$). Now let $\mathbb{B}^0A=A$ and iteratively define $\mathbb{B}^nA=\mathbb{B}(\mathbb{B}^{n-1}A)$. We denote by $\underline{E}A$ the resulting naive $G$-spectrum with $n$-th $G$-space 
    \[(\underline{E}A)_n=(\mathbb{B}^nA)(\mathbf{1}^+)\ .\]
\end{definition}

If $A$ is a cofibrant $\Gamma$-$G$-space, the simplicial space $\mathbf{q}^+ \mapsto A(\mathbf{p}^+ \wedge \mathbf{q}^+)$ is Reedy cofibrant. If $\varphi$ is injective, then $\mathbb{B}A(\varphi)$ is the realisation of the levelwise $G$-cofibrations $A(\varphi \wedge \id{\mathbf{q}^+})$ and hence a $G$-cofibration by \cite[Thm.~1.13]{paper:MayMerlingOsorno}. Thus $\mathbb{B}A$ is again cofibrant and, by induction, so is $\mathbb{B}^nA$ for all $n$. This machine produces a positive $\Omega$-$G$-spectrum. Its first structure map 
\[
    (\underline{E}A)_0=A(\textbf{1}^+) \rightarrow \Omega (\underline{E}A)_1
\] 
is a group completion by \cite[Prop.~2.19]{paper:MayMerlingOsorno}. The cofibrancy assumption on $A$ is no restriction: \cite[Prop.~2.12]{paper:MayMerlingOsorno} provides a functor $W$ on $\Gamma$-$G$-spaces such that $WA$ is cofibrant together with a natural levelwise $G$-homotopy equivalence $WA \to A$. In particular, $WA$ is special if $A$ is, and $WA(\mathbf{1}^+)^H \simeq A(\mathbf{1}^+)^H$ for all $H \le G$. We will therefore tacitly replace every special $\Gamma$-$G$-space $A$ by $WA$ before applying the machine.

%Note that this machine actually produces positive $\Omega$-$G$-spectra, and the first structure map $(\underline{E}A)_0=A(\textbf{1}^+) \rightarrow \Omega (\underline{E}A)_1$ is a group completion \cite[Prop. 2.18]{paper:MayMerlingOsorno}.

So far we have shown how to construct a positive $\Omega$-$G$-spectrum out of a commutative $\cI$-$\cO_G$-monoid. Our goal for the rest of the subsection is to show that for a \textit{stable} $\cI$-$\cO_G$-monoid $X$ (see Definition \ref{stable}), all of the information about the homotopy type of such a spectrum (except in level $0$) is contained in the $\cO_G$-space $X(\textbf{1})$.

The naive $G$-spectra we obtain as the output of the equivariant Segal machine (using $\Gamma$-$G$-spaces) are spectra indexed on the trivial universe, with
based $G$-spaces as levels and $G$-equivariant structure maps. Therefore homotopy groups and equivalences are defined as follows:
\begin{definition} \label{def:hpty_groups}
Let $n$ be an integer and $H \le G$ be a subgroup. The $n$-th equivariant homotopy group $\pi_n^H X$ of a naive $G$-spectrum $X$ is defined as
\[\pi_n^H(X) = \colim_{k \to \infty} \left[S^{n + k}, X_k\right]_{*,H} \]
where $H$ acts trivially on $S^{n+k}$. We say a morphism $f:X \rightarrow Y$ of naive $G$-spectra is an equivalence if the induced map $\pi_n^H(f):\pi_n^H(X) \rightarrow \pi_n^H(Y)$ is an isomorphism for all integers $n$ and all subgroups $H \le G$.

    %Fix a $G$-set universe $\cU$ and let $n$ be an integer and $H \le G$ be a subgroup. The $n$-th equivariant homotopy group $\pi_n^H X$ of a symmetric $G$-spectrum $X$ is defined as
    %\[\pi_n^H X= \colim_{M \in s_G(\cU)} [S^{n \sqcup M}, X(M)]^H\ ,\]
    %where $s_G(\cU)$ is the poset of finite $G$-subsets of $\cU$ (partially ordered by inclusion).
\end{definition}

\begin{lemma}
\label{+omega}
For every positive $\Omega$-$G$-spectrum $X$, every $n \ge 0$ and every subgroup $H \le G$ the map
\[\pi_n(\Omega X_1^H) \rightarrow \pi_n^H(X)\]
is an isomorphism.
\end{lemma}

\begin{proof}
Note that the term in the colimit on the right hand side of Definition~\ref{def:hpty_groups} is $\pi_n^H(\Omega^k X_k)$ and by the definition of a positive $\Omega$-$G$-spectrum we have that $\Omega X_1 \to \Omega^{k} X_{k}$ is an equivariant weak equivalence for $k \ge 1$.
%See \cite[Ex. 3.2]{paper:Hausmann} with $X(M)=X_1$.
\end{proof}

\begin{definition}
\label{stable}
    An $\cI$-$\cO_G$-space is called stable if all morphisms $\mathbf{m} \to \mathbf{n}$ in~$\cI$ with $m > 0$ are mapped to weak equivalences for all cosets~$G/H$.
\end{definition}

\begin{remark}
\label{rem_stable}
    Let $Z$ be a stable $\cI$-$\cO_G$-space. Then all structure maps of $Z(-)(G/H)$
    with nonempty source are weak equivalences, hence arbitrarily highly connected. B\"okstedt's approximation lemma
    \cite[Thm.~1.5]{paper:Bokstedt} therefore implies that the canonical map
    \[Z(\mathbf1)(G/H)\longrightarrow Z(G/H)_{h\mathcal I}\]
    is a weak equivalence.
\end{remark}

\begin{lemma}
\label{deloop}
Let $X, Y$ be two \textit{stable} commutative $\cI$-$\cO_G$-monoids such that $\pi_0(X(\mathbf{1})(G/H))$ and $\pi_0(Y(\mathbf{1})(G/H))$ are groups for all subgroups $H \le G$ and let $f: X \rightarrow Y$.
We have the following commutative diagram
% https://q.uiver.app/#q=WzAsNixbMCwwLCJcXHBpX3trfV57SH0oRUhfe1xcbWF0aGNhbHtJfX1YKSJdLFsyLDAsIlxccGlfe2t9XntIfShFSF97XFxtYXRoY2Fse0l9fVkpIl0sWzAsMSwiXFxwaV97a30gKChFSF97XFxtYXRoY2Fse0l9fVhfezF9KV57SH0pIl0sWzIsMSwiXFxwaV97a30gKChFSF97XFxtYXRoY2Fse0l9fVlfezF9KV57SH0pIl0sWzAsMiwiXFxwaV97a30oWCgxKV57SH0pIl0sWzIsMiwiXFxwaV97a30oWSgxKV57SH0pIl0sWzAsMV0sWzIsMCwiXFxzaW1lcSJdLFsyLDNdLFszLDEsIlxcc2ltZXEiLDJdLFs0LDVdLFs1LDMsIlxcc2ltZXEiLDJdLFs0LDIsIlxcc2ltZXEiXV0=
\[\begin{tikzcd}
	{\pi_{n}^{H}(\underline{E}H_{\mathcal{I}}X)} && {\pi_{n}^{H}(\underline{E}H_{\mathcal{I}}Y)} \\
	{\pi_{n+1} (((\underline{E}H_{\mathcal{I}}X)_1)^{H})} && {\pi_{n+1} (((\underline{E}H_{\mathcal{I}}Y)_1)^{H})} \\
	{\pi_{n}(X(\mathbf{1})(G/H))} && {\pi_{n}(Y(\mathbf{1})(G/H))}
	\arrow[from=1-1, to=1-3]
	\arrow["\simeq", from=2-1, to=1-1]
	\arrow[from=2-1, to=2-3]
	\arrow["\simeq"', from=2-3, to=1-3]
	\arrow[from=3-1, to=3-3]
	\arrow["\simeq"', from=3-3, to=2-3]
	\arrow["\simeq", from=3-1, to=2-1]
\end{tikzcd}\]
for every $n \ge 0$ and every $H \le G$. In particular, $f$ induces an isomorphism on $\pi_n^H$ of the corresponding $G$-spectra if and only if the map
\[\pi_{n}(X(\mathbf{1})(G/H)) \rightarrow \pi_{n}(Y(\mathbf{1})(G/H))\]
is an isomorphism.
\end{lemma}

\begin{proof}
The two upper vertical arrows are isomorphisms by Lemma \ref{+omega}. 

Fix $H \le G$ and let $Z \in \{X,Y\}$. By construction we have
\[((\underline{E}H_{\mathcal{I}}Z)_0)^H = (H_{\mathcal{I}}Z(\textbf{1}^+))^H \simeq Z(G/H)_{h\cI}\ ,\]
so Remark~\ref{rem_stable} provides monoid isomorphisms
\[ \pi_0(Z(\mathbf1)(G/H)) \cong \pi_0(Z(G/H)_{h\cI}) \cong \pi_0(((\underline{E}H_\cI Z)_0)^H)\ .\]
Hence $(\underline{E}H_\cI Z)_0$ is grouplike and by the group completion theorem
\[((\underline{E}H_\cI Z)_0)^H \simeq (\Omega (\underline{E}H_{\cI} Z)_1)^H \ .\]
Therefore for every $n \ge 0$
\[ \pi_n(Z(\mathbf{1})(G/H)) \cong \pi_n(((\underline{E}H_\cI Z)_0)^H) \cong \pi_{n+1}(((\underline{E}H_\cI Z)_1)^H)\]
and these isomorphisms are natural in $Z$, so they provide the two lower vertical arrows and make the diagram commute.
\end{proof}

\subsection{Preliminaries from operator algebras.} \label{prel_2}

We will look at $G$-$C^*$-alge\-bras that arise from infinite tensor products of $\Z/n\Z$-representations. Such algebras belong to the class of strongly self-absorbing $C^*$-algebras \cite{paper:TomsWinter}.
\begin{definition}
\label{ssa}
A separable, unital $C^*$-algebra $D$ is called \textit{strongly self-absorbing} if it tensorially absorbs itself, via an isomorphism $\psi \colon D \rightarrow D \otimes D$, in such a way that $\psi$ is asymptotically unitarily equivalent\footnote{The original definition uses approximate unitary equivalences. This is equivalent to the stated one by \cite[Thm.~2.2]{paper:DadarlatWinter}, \cite{paper:Winter-ZStable} and \cite[Thm.~2.3]{paper:Hua}.} to the left tensor embedding $d \mapsto d \otimes 1$.
\end{definition}

%To put this class into context note that some properties that are essential in the classification programme of nuclear, separable, simple $C^*$-algebras are related to tensorial absorption of strongly self-absorbing ones:
%\begin{enumerate}[(i)]
    %\item If $A$ is separable, simple and nuclear, then $A$ is purely infinite if and only if $A \otimes \Cuntz{\infty} \cong A$ \cite[Thm.~3.15]{paper:KirchbergPhillips}, \cite[Thm.~7.2.6 (ii)]{book:RordamStormer}. Here, $\Cuntz{\infty}$ denotes the infinite Cuntz algebra, which is strongly self-absorbing (and therefore purely infinite).  
    %\item Separable, simple, unital, nuclear $C^*$-algebras in the UCT class with at most one trace are classified up to $\mathcal{Z}$-stability by their ordered $K$-theory \cite[Cor.~E]{paper:TWW}. Here, $\mathcal{Z}$ denotes the Jiang-Su algebra, which is strongly self-absorbing and ``up to $\mathcal{Z}$-stability'' means after taking a tensor product with $\mathcal{Z}$.
    %\item Let $A$ be a simple, separable, unital and nuclear $C^*$-algebra. Then $A \otimes \Cuntz{2} \cong \Cuntz{2}$ \cite[Cor.~3.8]{paper:KirchbergPhillips}. Here, $\Cuntz{2}$ is the Cuntz algebra on two generators, which is also strongly self-absorbing.  
%\end{enumerate}
%Apart from $\C$, $\mathcal{Z}$, $\Cuntz{\infty}$ and $\Cuntz{2}$ other 
Important examples of strongly self-absorbing $C^*$-algebras that play a crucial role in the classification programme for nuclear simple separable $C^*$-algebras include the infinite Cuntz algebra $\Cuntz{\infty}$ and the Jiang-Su algebra~$\mathcal{Z}$. But the ones that are most important for this paper arise from UHF-algebras, which are infinite tensor products of matrix algebras. Let $r \in \N$ and consider the sequence
\[ 
   \C \to \dots \to M_r(\C)^{\otimes k} \to M_r(\C)^{\otimes (k+1)} \to M_r(\C)^{\otimes (k+2)} \to \dots\ ,
\]
where the connecting homomorphisms are given by $T \mapsto T \otimes 1_r$. Denote the colimit of this sequence in the category of unital $C^*$-algebras by $M_r^{\otimes \infty}$. If $p_1, \dots, p_m$ are the prime factors of $r$, then we have 
\[
    M_r^{\otimes \infty} \cong M_{p_1}^{\otimes \infty} \otimes \dots \otimes M_{p_m}^{\otimes \infty}\ .
\]
For an arbitrary subset $P$ of the prime numbers we therefore define (by slight abuse of notation) 
\[
    M_P^{\otimes \infty} = \bigotimes_{p \in P} M_p^{\otimes \infty}\ .
\]
If $P$ is the set of all prime numbers, the resulting algebra is called the universal UHF-algebra, usually denoted by $\mathcal{Q}$, because $K_0(\mathcal{Q}) \cong \Q$. 
%A list of all strongly self-absorbing $C^*$-algebras in the UCT-class can be found in Figure~\ref{fig:ssa_diamond}.
%\begin{figure}[htp]
%    \centering
%    \begin{tikzcd}[column sep=1.4cm,row sep=0.2cm]
%        & \mathcal{Z} \ar[dd] \ar[r] & M_P^{\otimes \infty} \ar[r] \ar[dd] & \mathcal{Q} \ar[dd] \ar[dr] \\ 
%        \C \ar[ur] \ar[dr] &&&& \Cuntz{2} \\
%        & \Cuntz{\infty} \ar[r] & M_P^{\otimes \infty} \otimes \Cuntz{\infty} \ar[r] & \mathcal{Q} \otimes \Cuntz{\infty} \ar[ur]
%    \end{tikzcd}
%    \caption{\label{fig:ssa_diamond}The list of strongly self-absorbing $C^*$-algebras in the UCT class. An arrow indicates tensorial absorption (e.g.\ $\Cuntz{\infty} \otimes \mathcal{Z} \cong \Cuntz{\infty}$). }
%    \label{fig:ssa_diamond}
%\end{figure}
%\pagebreak{10}

Let $D$ be a strongly self-absorbing $C^*$-algebra and consider the functor $X \mapsto K_*(C(X) \otimes D)$. The tensor product induces a ring structure on $K_*(C(X) \otimes D)$, given by
\[
\begin{tikzcd}[column sep=0.6cm,every matrix/.append style={nodes={font=\small}}]
    K_*(C(X) \otimes D) \times K_*(C(X) \otimes D) \ar[r] & K_*(C(X \times X) \otimes D) \ar[r,"\Delta^*"] & K_*(C(X) \otimes D)
\end{tikzcd}
\]
where the first arrow takes the tensor product and applies the inverse isomorphism $\psi^{-1} \colon D \otimes D \to D$, and $\Delta \colon X \to X \times X$ denotes the diagonal map. The multiplication does not depend on the choice of $\psi$ because $\Aut{D}$ is contractible \cite[Thm. 2.3]{paper:DP-Dixmier} and $K$-theory is homotopy invariant. In fact, one can check that the assignment $X \rightarrow K_*(C(X) \otimes D)$ defines a multiplicative cohomology theory on finite CW-complexes \cite[Sec.~4.1]{paper:DP-Units}.

Because of their infinite tensor product structure it is easy to construct interesting group actions on UHF-algebras. Let $G$ be a discrete group and let~$\rho \colon G \to U(V)$ be a finite-dimensional unitary $G$-representation, then $G$ acts by conjugation on $\Endo{V}$. This $G$-action extends to the UHF-algebra
\(
    D = \Endo{V}^{\otimes \infty}.
\)
Note that while the action of $G$ on $\Endo{V}$ is inner, this is no longer the case for the induced action on $D$, where the action is in general only approximately inner. %In the next section we will combine the above two observations and construct a $G$-spectrum $KU^D$ representing the equivariant cohomology theory $X \mapsto K_*^G(C(X) \otimes D)$.

\section{The $G$-spectrum of units of $KU^D$-theory}
\label{sec3}

\subsection{The spectrum $KU^D$}
\label{spectrum}
In this section we will construct a $G$-equivariant commutative symmetric ring spectrum $KU^D$ representing the equivariant cohomology theory $X \mapsto K_*^G(C(X) \otimes D)$, for $G = \Z/n\Z$. This spectrum is associated to a certain functor taking values in graded $G$-$C^*$-algebras (with grading preserving action), which we call a coefficient system. 

\begin{definition}
    Let $\Sigma$ be the category of finite sets and bijective maps. This is a symmetric monoidal category with respect to the coproduct $\sqcup$. Let $G$ be a finite group and let $\GCStar$ be the category of graded $G$-$C^*$-algebras and equivariant grading-preserving $*$-homomorphisms. This is also a symmetric monoidal category with respect to the graded minimal tensor product $\otimes$. A \textit{coefficient system} is a symmetric monoidal functor 
    \[
        C \colon (\Sigma, \sqcup) \to (\GCStar, \otimes)\ .
    \]
    We call $C$ \emph{stable} if for every nonempty finite set $S$ we have $C(S) \otimes \bK \cong C(S)$ as graded $G$-$C^*$-algebras.
    
    The tensor product of two coefficient systems $C_1$ and $C_2$ is defined to be  
    \[
        (C_1 \otimes C_2)(S) = C_1(S) \otimes C_2(S)
    \]
    for a finite set $S$. Let $\eta^i_{S,T} \colon C_i(S \sqcup T) \to C_i(S) \otimes C_i(T)$. The corresponding natural isomorphism for $C_1 \otimes C_2$ is 
    \[
    \begin{tikzcd}
        (C_1 \otimes C_2)(S \sqcup T) =
         C_1(S \sqcup T) \otimes C_2(S \sqcup T) \arrow[d,"\eta^1_{S,T} \otimes \eta^2_{S,T}"]  \\
         C_1(S) \otimes C_1(T) \otimes C_2(S) \otimes C_2(T) \arrow[d,"\id{} \otimes \text{flip} \otimes \id{}"] \\
         (C_1 \otimes C_2)(S) \otimes (C_1 \otimes C_2)(T) = C_1(S) \otimes C_2(S) \otimes C_1(T) \otimes C_2(T)
    \end{tikzcd}    
    \]
    where flip denotes the symmetry of $\GCStar$ (taking the grading into account).  
\end{definition}

Our main example of a coefficient system is constructed as follows: let $V$ be a finite-dimensional complex inner product space and let $\rho \colon G \to U(V)$ be a unitary representation of $G$. Let
\begin{equation} \label{eqn:D}
    D = \Endo{V}^{\otimes \infty}
\end{equation}
be the infinite UHF-algebra associated to $V$. The infinite tensor product of the adjoint action of $G$ on $\Endo{V}$ turns this into a $G$-$C^*$-algebra, which we consider to be trivially graded. 

To define the stabilisation of $D$, we need the compact operators, which we turn into a $G$-$C^*$-algebra as follows: let $H_0$ be an infinite-dimensional separable Hilbert space and define $H_G = \ell^2(G) \otimes H_0$, where $\ell^2(G)$ is the left regular representation of $G$. Let $\bK = \bK(H_G)$ be the compact operators on~$H_G$. This is a $G$-$C^*$-algebra where $G$ acts on $\bK(\ell^2(G) \otimes H_0)$ via the adjoint action. We consider $D \otimes \bK$ equipped with the tensor product action of $G$, which we denote by $\alpha$, and the trivial grading. We define 
\[
    (D \otimes \bK)(S) = \bigotimes_{s \in S} D \otimes \bK
\]
and note the natural isomorphism
\[
    (D \otimes \bK)(S \sqcup T) = \bigotimes_{s \in S \sqcup T} D \otimes \bK \cong \left(\bigotimes_{s \in S} D \otimes \bK\right) \otimes \left(\bigotimes_{s \in T} D \otimes \bK\right)
\]
given by reordering the tensor factors. Equipped with this natural isomorphism $S \mapsto (D \otimes \bK)(S)$ is a trivially graded stable coefficient system, where the group acts on $D$ and $\bK$ as described above.

To define the ring spectrum we will need two non-trivially graded $C^*$-algebras that feature in the definition of $KU^D$. The first one is the graded suspension algebra $\grS = C_0(\R)$ equipped with the grading by odd and even functions. This is not only an algebra, but also a coassociative and cocommutative coalgebra with respect to the comultiplication
\(
    \Delta \colon \grS \to \grS \otimes \grS \ ,
\)
where the tensor product in the codomain is the graded one. 

The second $\Z/2\Z$-graded $C^*$-algebra needed in the construction is the complex Clifford algebra. Let $S \in \Sigma$ be a finite set and let $\C[S]$ be the finite-dimensional complex inner product space equipped with the symmetric bilinear form 
\[
    b(v, w) = \sum_{s \in S} v_s w_s\ ,
\]
where $v_s, w_s$ denote the components with respect to the orthonormal basis~$S$. The Clifford algebra $\Cliff{S}$ is defined to be the quotient
\[
    \Cliff{S} = T(\C[S])/I(\C[S],b)
\]
where $T(W) = \bigoplus_{n \in \N_0} W^{\otimes n}$ is the tensor algebra and $I(W,b)$ is the ideal generated by $v \otimes w + w \otimes v - 2b(v,w)\,1$ for $v,w \in W$ where $b$ is a bilinear form on $W$. To define the grading we declare the elements $w \in W$ to be odd. There is a natural isomorphism of graded algebras
\[
    \Cliff{S \sqcup T} \cong \Cliff{S} \otimes \Cliff{T}\ ,
\]
which turns $S \mapsto \Cliff{S}$ into a coefficient system with trivial $G$-action. For $n \in \N$ let $\Cln{n} = \Cliff{\{1,\dots,n\}}$. 

In order to describe the suspension isomorphism and Bott periodicity in the spectral picture we define $\hat{\eta}_1 \in \hom_{\text{gr}}(\grS, C_0(\R) \otimes \Cln{1})$ to be the Bott element as in \cite[Def.~1.26]{book:HigsonGuentner}. Taking the adjoint map we can view this element as a base-point preserving continuous map 
\[
    \tilde{\eta}_1 \colon S^1 \to \hom_{\text{gr}}(\grS, \Cln{1})
\]
and extend it to 
\[
    \eta_1 \colon S^1 \to \hom_{\text{gr}}(\grS, \Cln{1} \otimes D \otimes \bK) \quad , \quad z \mapsto \tilde{\eta}_1(z) \otimes 1_D \otimes e
\]
where $e\in\mathbb K$ is a rank-one projection onto a one-dimensional trivial $G$-subrepresentation of $H_G$.
Since $1_D\otimes e$ is fixed by the group action on $D \otimes \bK$, the map $\eta_1$
is $G$-equivariant. Now we define the spectrum $KU^D$ as follows:

\begin{definition} \label{def:KUD}
Let $G$ be a finite group, $V$ be a unitary $G$-re\-pre\-sen\-tation and let $D$ be the associated UHF $G$-$C^*$-algebra as defined in \eqref{eqn:D}. Let $(KU^D_n)_{n \in \N_0}$ denote the sequence of $G$-spaces 
\[ 
    KU^D_{n}=\hom_{\text{gr}}(\grS,\, (\C\ell \otimes D \otimes \bK)(\{1,\dots,n\})\,) \ ,
\]
where $S \mapsto (\C\ell \otimes D \otimes \bK)(S)$ is the coefficient system defined above. Notice that there is a natural $(\Sigma_n \times G)$-action on $(\C\ell \otimes D \otimes \bK)(\{1,\dots,n\})$ by functoriality. The spaces $KU^D_n$ are pointed at the zero homomorphism and can be equipped with a multiplication as follows:
\[
    \mu_{m,n} \colon KU^D_m \wedge KU^D_n \to KU^D_{m+n} \quad, \quad \varphi \wedge \psi \mapsto (\varphi \otimes \psi) \circ \Delta\ .
\]
There are corresponding unit maps $\eta_n \colon S^n \to KU^D_n$ defined inductively by 
\(
    \eta_n = (\eta_{n-1} \otimes \eta_1) \circ \Delta
\)
with $\eta_1$ as above and $\eta_0:S^0\to KU_0^D$ sending the non-basepoint to $\operatorname{ev}_0:\widehat S\to\mathbb C$.
\end{definition}

\begin{remark}
    Non-equivariant versions of the spectrum $KU^{\cD}$ have been defined in \cite{paper:DP-Units} for all strongly self-absorbing $C^*$-algebras $\cD$. In the UCT-case the spectrum $KU^{\cD}$ represents topological $K$-theory with coefficients in $K_0(\cD)$, i.e. a localisation of $K$-theory. For instance, 
    \[KU^{\C} \simeq KU^{\mathcal{Z}} \simeq KU^{\mathcal{O}_{\infty}} \simeq KU \qquad \text{ and } \qquad KU^{M_n^{\otimes \infty}} \simeq KU\!\left[\tfrac{1}{n}\right] \]
    where $KU$ denotes the usual complex topological $K$-theory spectrum.
\end{remark}

 Given two graded $G$-$C^*$-algebras $A$ and $B$ denote by $[A,B]_G$ the $G$-homotopy classes of grading preserving $G$-equivariant $*$-homomorphisms $A \to B$. Let $\hat{\C}^2$ be the graded Hilbert space with $\hat{\C}^2_+ = \C = \hat{\C}^2_-$ and let $\hat{H}_G = \hat{\C}^2 \otimes H_G = \hat{\C}^2 \otimes \ell^2(G) \otimes H_0$. Let $\widehat{\bK} = \bK(\hat{H}_G)$. This is also a graded $G$-$C^*$-algebra. The following lemma is an equivariant generalisation of the converse functional calculus in \cite{paper:Trout}. In fact, the equivariant case can be reduced to the non-equivariant one via the Green-Julg theorem. A proof along these lines (which works for the much more general setting of proper locally compact groupoids with Haar system) can be found for example in \cite[Prop.~3.1]{paper:FuncEquivKSpec}. 
\begin{lemma} \label{lem:htpy_sets}
    Let $G$ be a finite group and let $B$ be a graded $G$-$C^*$-algebra. Then the converse functional calculus explained in \cite[Thm.~4.7]{paper:Trout} provides an isomorphism of groups 
    \[
        [\grS, B \otimes \widehat{\bK}]_G \cong KK_G(\C,B)\ ,
    \]
    which is natural in $B$. 
\end{lemma}

We are now in the position to prove the main theorem of this section.

\begin{theorem} \label{KU}
The spaces $KU^D_{*}$ together with the maps $\mu_{*,*}$ and $\eta_{*}$ form a commutative symmetric $G$-ring spectrum (in the sense of \cite[Def.~2.2]{paper:Hausmann} for the trivial universe), with coefficients (for $H \le G)$
\[ 
    \pi_{m}^{H}(KU^D) \cong K_{m}^{H}(D)\ .
\]
\end{theorem}

\begin{proof}
The proof that the sequence of spaces $(KU^D_m)_{m \in \N_0}$ forms a commutative symmetric $G$-ring spectrum is very similar to the non-equivariant case discussed in \cite[Thm.~4.2]{paper:DP-Units}. Therefore we omit the details and just point out that $KU^D_m$ comes equipped with a $G$-action, which is trivial everywhere except on $D \otimes \bK$. In particular, it commutes with the $\Sigma_m$-action.

To compute $\pi_m^H(KU^D)$ we need to determine
\(
    [S^{m + k}, KU^D_k]^H_\ast
\) for $k \in \N_0$. 
Observe that 
\[
    KU^D_k \cong \hom_{\text{gr}}\left(\grS, \Cln{k} \otimes (D \otimes \bK)^{\otimes k}\right) 
\]
and therefore 
\begin{equation} \label{eqn:htpy_groups}
    \left[S^{m +k}, KU^D_k\right]_{*,H} \cong \left[\grS, C_0(\R^m) \otimes C_0(\R^k, \Cln{k}) \otimes (D\otimes\bK)^{\otimes k}\right]_H\ .
\end{equation}
The grading on $\widehat{\bK}$ is generated by an inner automorphism implemented by the $G$-equivariant operator $\gamma$ on $\hat{H}_G$. If $\varepsilon$ denotes the odd generator of $\C\ell_1$, then $T \otimes 1 \mapsto T \otimes \varepsilon^{\lvert T \rvert}$ for homogeneous $T$ of degree $\lvert T \rvert$ and $1 \otimes \varepsilon \mapsto \gamma \otimes \varepsilon$ define a $*$-isomorphism $\widehat{\bK} \otimes \C\ell_1 \to \bK \otimes \C\ell_1$ of graded $G$-$C^*$-algebras, where $\bK$ on the right-hand side is trivially graded. This gives us:
%From \cite[Cor.~14.5.3]{book:Blackadar} we deduce the following isomorphisms of graded $G$-$C^*$-algebras:
\[
    \Cln{k} \otimes \widehat{\bK} \cong \Cln{k-1} \otimes \Cln{1} \otimes \widehat{\bK} \cong \Cln{k-1} \otimes \Cln{1} \otimes \bK \cong \Cln{k} \otimes \bK\ ,
\]
where $\bK$ on the right hand side denotes the trivially graded compact operators. By Lemma~\ref{lem:htpy_sets} we therefore can express the right hand side of \eqref{eqn:htpy_groups} as equivariant $K$-theory groups as follows:
\begin{align*}
    & KK_H(\C, C_0(\R^m) \otimes C_0(\R^k, \Cln{k}) \otimes (D\otimes\bK)^{\otimes k} ) \\
    \cong\ & KK_H(\C, C_0(\R^m) \otimes (D\otimes\bK)^{\otimes k} ) \cong K_m^H( (D\otimes \bK)^{\otimes k} )\ ,
\end{align*}
where the first isomorphism is a consequence of Bott periodicity in equivariant $KK$-theory. For $l \geq k$ the above identification fits into the following commutative diagram
\[
    \begin{tikzcd}
        \left[S^{m + k}, KU^D_k\right]_{*,H} \arrow[d] \arrow[r,"\cong"] & K_m^H( (D\otimes \bK)^{\otimes k} ) \arrow[d] \\
        \left[S^{m + l}, KU^D_l\right]_{*,H} \arrow[r,"\cong" below] & K_m^H( (D\otimes \bK)^{\otimes l} )         
    \end{tikzcd}
\]
where the right hand vertical arrow is induced by inclusions $(D \otimes \bK)^{\otimes k} \to (D \otimes \bK)^{\otimes l}$ that introduce tensor factors $1_D \otimes e$. All of these maps are isomorphisms in $K$-theory. Hence, 
\[
    \pi_m^H(KU^D) = \colim_{k \to \infty} \left[S^{m + k}, KU^D_k\right]_{*,H} \cong K_m^H(D)\ . \qedhere
\]
\end{proof}

\begin{remark}
With only minor modifications one can adapt the above proof to see that $KU^D_m$ extends to a commutative symmetric $G$-ring spectrum for an arbitrary choice of $G$-universe in the sense of \cite[Def.~2.2]{paper:Hausmann}. We omit the details here, since we will only make use of $KU^D$ as a naive $G$-spectrum. 
\end{remark}

\subsection{Equivariant units of $KU^D$}
\label{equnits}

Given a (commutative) symmetric $G$-ring spectrum $R$ with unit maps $\eta_n$ and multiplication maps $\mu_{m,n}$, %the assignment
\[\textbf{n} \mapsto (G/H \mapsto (\Omega^{n}R_{n})^H)\]
defines a (commutative) $\cI$-$\cO_G$-monoid $\Omega^{\bullet} R$. 
A morphism $\alpha: \textbf{m} \rightarrow \textbf{n}$ in~$\cI$ induces a unique permutation $\bar{\alpha}: \textbf{n} = \textbf{l} \sqcup \textbf{m} \rightarrow \textbf{n}$, which acts like $\alpha$ on the last $m$ entries and is monotonic on the first $l$ entries. This allows to define a map $(\Omega^{m}R_{m})^H \rightarrow (\Omega^{n}R_{n})^H$ by sending $f \in (\Omega^{m}R_{m})^H$ to the composition
\[S^{n} \xrightarrow{\bar{\alpha}^{-1}} S^{n}=S^{l} \wedge S^{m} \xrightarrow{\eta_{l} \wedge f} R_{l} \wedge R_{m} \xrightarrow{\mu_{l,m}} R_{n} \xrightarrow{\bar{\alpha}} R_{n}\ ,\]
thus yielding functoriality of $\Omega^{\bullet} R$ with respect to injective maps.

Non-equivariantly, the group of units of $R$ is classically defined by the following pullback (see \cite[Def.~11.3]{paper:Lind}):
% https://q.uiver.app/#q=WzAsNCxbMCwwLCJcXHRleHR7R0x9XzFSIl0sWzAsMSwiXFx0ZXh0e0dMfV8xKFxccGlfMChSKSkiXSxbMSwwLCJcXE9tZWdhXntcXGluZnR5fVIiXSxbMSwxLCJcXHBpXzAoUikiXSxbMiwzXSxbMywxLCIiLDIseyJzdHlsZSI6eyJib2R5Ijp7Im5hbWUiOiJub25lIn0sImhlYWQiOnsibmFtZSI6Im5vbmUifX19XSxbMCwxXSxbMCwyLCIiLDEseyJzdHlsZSI6eyJ0YWlsIjp7Im5hbWUiOiJob29rIiwic2lkZSI6InRvcCJ9fX1dLFsxLDMsIiIsMSx7InN0eWxlIjp7InRhaWwiOnsibmFtZSI6Imhvb2siLCJzaWRlIjoidG9wIn19fV1d
\begin{equation} 
\label{eqn:GL1_diag}
\begin{tikzcd}
	{\text{GL}_1R} & {\Omega^{\infty}R} \\
	{\text{GL}_1(\pi_0(R))} & {\pi_0(R)}
	\arrow[hook, from=1-1, to=1-2]
	\arrow[from=1-1, to=2-1]
	\arrow[from=1-2, to=2-2]
	\arrow[hook, from=2-1, to=2-2]
\end{tikzcd}
\end{equation}

In the equivariant setting, one requires this condition to hold simultaneously for all fixed points.
Hence we define the $\cI$-$\cO_G$-\textit{space of units} of~$R$, which we denote by GL$_1^{\bullet}R$, by the pullback diagram analogous to \eqref{eqn:GL1_diag} in the category of $\cI$-$\cO_G$-spaces. 
Unraveling the definition shows that it encodes the information about invertible elements for all fixed points: more precisely, $(\text{GL}_1^{\bullet}R(\textbf{n}))(G/H)$ consists of those elements of $(\Omega^nR_n)^H$ that map to units on $\pi_0$, where $\pi_0$ denotes the $\cO_G$-set defined by $\pi_0(X)(G/H)=\pi_0^H(X)$. For semistable commutative symmetric $G$-ring spectra this has the correct homotopy type (see \cite[Rem.~11.4]{paper:Lind}), so in particular this is the case for $KU^D$. The machinery developed in Section \ref{prel_1} allows us to define the special $\Gamma$-$G$-space of units $H_{\cI}\text{GL}_1^{\bullet}R$, hence its associated naive $G$-spectrum via Definition \ref{def_E}.

\begin{definition} \label{def:spec_units}
The $G$-\textit{spectrum of units} of a commutative symmetric $G$-ring spectrum~$R$ is the naive $G$-spectrum
\[gl_{1}R=\underline{E}H_{\cI}\text{GL}_1^{\bullet}R\ .\]
\end{definition}

\begin{remark}
    A construction of the units of equivariant ring spectra in terms of equivariant $\Gamma$-spaces (namely $\Gamma_G$-spaces, a slight variation of $\Gamma$-$G$-spaces) has been carried out in \cite[Sec.~9]{paper:Santhanam}.
    We now show that, for a special $\Gamma$-$\cO_G$-space arising from a stable commutative $\cI$-$\cO_G$-monoid $\Omega^{\bullet}R$ for a symmetric commutative $G$-ring spectrum $R$, our machine and the one proposed in \cite{paper:Santhanam} define weakly equivalent objects. Let $UR$ denote the $\Gamma$-$\cO_G$-space of units considered in \cite{paper:Santhanam}. It is defined by the  pullback 
   % https://q.uiver.app/#q=WzAsNCxbMCwwLCJcXHdpZGV0aWxkZXtcXHRleHR7R0x9XzF9XntcXGJ1bGxldH1SIl0sWzAsMSwiXFx0ZXh0e0dMfV8xKFxccGlfMChIX3tcXG1hdGhjYWx7SX19XFxPbWVnYV57XFxidWxsZXR9UikpIl0sWzEsMCwiSF97XFxtYXRoY2Fse0l9fVxcT21lZ2Fee1xcYnVsbGV0fVIiXSxbMSwxLCJcXHBpXzAoSF97XFxtYXRoY2Fse0l9fVxcT21lZ2Fee1xcYnVsbGV0fVIpIl0sWzIsM10sWzMsMSwiIiwyLHsic3R5bGUiOnsiYm9keSI6eyJuYW1lIjoibm9uZSJ9LCJoZWFkIjp7Im5hbWUiOiJub25lIn19fV0sWzAsMV0sWzAsMiwiIiwxLHsic3R5bGUiOnsidGFpbCI6eyJuYW1lIjoiaG9vayIsInNpZGUiOiJ0b3AifX19XSxbMSwzLCIiLDEseyJzdHlsZSI6eyJ0YWlsIjp7Im5hbWUiOiJob29rIiwic2lkZSI6InRvcCJ9fX1dXQ==
\[\begin{tikzcd}[row sep=0.5cm]
	{UR} & {H_{\mathcal{I}}\Omega^{\bullet}R} \\
	{\text{GL}_1(\pi_0(H_{\mathcal{I}}\Omega^{\bullet}R))} & {\pi_0(H_{\mathcal{I}}\Omega^{\bullet}R)}
	\arrow[hook, from=1-1, to=1-2]
	\arrow[from=1-1, to=2-1]
	\arrow[from=1-2, to=2-2]
	\arrow[hook, from=2-1, to=2-2]
\end{tikzcd}\]
where $H_{\mathcal{I}}\Omega^{\bullet}R$ is the special $\Gamma$-$\cO_G$-space associated to $\Omega^{\bullet}R$. Note that $H_{\cI}GL_1^\bullet R$ also fits into a homotopy pullback diagram
\[
    \begin{tikzcd}[row sep=0.5cm]
        H_{\cI}GL_1^\bullet R \ar[r] \ar[d] & H_{\cI} \ar[d] \Omega^\bullet R \\
        \pi_0(H_{\cI}GL_1^\bullet R) \ar[r] & \pi_0(H_{\cI}\Omega^\bullet R)
    \end{tikzcd}
\]
The map of $\Gamma$-$G$-spaces $H_\cI GL_1^\bullet R \to H_\cI \Omega^{\bullet} R$ is an isomorphism on all higher homotopy groups. Hence, it suffices to show that
\[
\pi_0(UR(\textbf{1}^+)) \cong \pi_0((H_{\cI}GL_1^\bullet R)(\mathbf{1}^+)) \ .
\] 
By stability we have $(H_{\cI}\text{GL}_1^{\bullet}R)(\textbf{1}^+) = (\text{GL}_1^{\bullet}R)_{h\cI} \simeq \text{GL}_1^{\bullet}R(\textbf{1})$ and therefore
\[\pi_0((H_{\cI}\text{GL}_1^{\bullet}R)(\textbf{1}^+)) \cong 
\pi_0(\text{GL}_1^{\bullet}R(\textbf{1})) \cong \text{GL}_1(\pi_0(\Omega^{\infty} R))\ .
\]
On the other hand, since $(H_{\cI}\Omega^{\bullet}R)(\textbf{1}^+)=\Omega^{\bullet}R_{h\cI} \simeq \Omega^{\infty}R$, we have that $UR(\textbf{1}^+)$ is defined by
% https://q.uiver.app/#q=WzAsNCxbMCwwLCJcXHdpZGV0aWxkZXtcXHRleHR7R0x9XzF9XntcXGJ1bGxldH1SKFxcdGV4dGJmezF9KSJdLFswLDEsIlxcdGV4dHtHTH1fMShcXHBpXzAoXFxPbWVnYV57XFxpbmZ0eX1SKSkiXSxbMSwwLCJcXE9tZWdhXntcXGluZnR5fVIiXSxbMSwxLCJcXHBpXzAoXFxPbWVnYV57XFxpbmZ0eX1SKSJdLFsyLDNdLFszLDEsIiIsMix7InN0eWxlIjp7ImJvZHkiOnsibmFtZSI6Im5vbmUifSwiaGVhZCI6eyJuYW1lIjoibm9uZSJ9fX1dLFswLDFdLFswLDIsIiIsMSx7InN0eWxlIjp7InRhaWwiOnsibmFtZSI6Imhvb2siLCJzaWRlIjoidG9wIn19fV0sWzEsMywiIiwxLHsic3R5bGUiOnsidGFpbCI6eyJuYW1lIjoiaG9vayIsInNpZGUiOiJ0b3AifX19XV0=
\[\begin{tikzcd}[row sep=0.5cm]
	{UR(\textbf{1}^+)} & {\Omega^{\bullet}R_{h\cI}} \\
	{\text{GL}_1(\pi_0(\Omega^{\infty}R))} & {\pi_0(\Omega^{\infty}R)}
	\arrow[hook, from=1-1, to=1-2]
	\arrow[from=1-1, to=2-1]
	\arrow[from=1-2, to=2-2]
	\arrow[hook, from=2-1, to=2-2]
\end{tikzcd}\]
Hence by definition 
\[
\pi_0(UR(\textbf{1}^+)) \cong \text{GL}_1(\pi_0(\Omega^{\infty}R)) \cong \pi_0(H_{\cI}\text{GL}_1^{\bullet}R(\textbf{1}^+))\ .
\]
\end{remark} 

We will focus on the equivariant units of the $G$-ring spectrum $KU^D$ constructed in Section \ref{spectrum}. The following result is basically a consequence of Bott periodicity.

\begin{lemma}
The $\cI$-$\cO_G$-space $\text{GL}_1^{\bullet}KU^D$ is stable.
\end{lemma}

\begin{proof}
We need to check that all morphisms $\mathbf{m} \to \mathbf{n}$ in~$\cI$ with $m > 0$ are mapped to weak equivalences for all cosets $G/H$. Let $H \le G$. Note that for $\pi_0$ we have
\[ \pi_{0} ((\text{GL}_1^{\bullet}KU^D)(\textbf{m})(G/H)) \cong \pi_{0}((\text{GL}_1^{\bullet}KU^D)(\textbf{n})(G/H)) \cong \text{GL}_1(K_0^H(D))\]
for all $m,n >0$, whereas for $k > 0$ we have the following commutative diagram:
% https://q.uiver.app/#q=WzAsNCxbMCwwLCJcXHBpX2soXFx0ZXh0e0dMfV8xXntcXGJ1bGxldH1SKShcXHRleHRiZnttfSkoRy9IKSJdLFswLDEsIlxccGlfa1xcT21lZ2Fee219S1VeRF9tIl0sWzEsMCwiXFxwaV9rKFxcdGV4dHtHTH1fMV57XFxidWxsZXR9UikoXFx0ZXh0YmZ7bn0pKEcvSCkiXSxbMSwxLCJcXHBpX2tcXE9tZWdhXntufUtVXkRfbiJdLFsyLDMsIlxcY29uZyJdLFszLDEsIiIsMix7InN0eWxlIjp7ImJvZHkiOnsibmFtZSI6Im5vbmUifSwiaGVhZCI6eyJuYW1lIjoibm9uZSJ9fX1dLFswLDEsIlxcY29uZyIsMl0sWzAsMl0sWzEsMywiXFxjb25nIiwyXV0=
\[\begin{tikzcd}
	{\pi_k((\text{GL}_1^{\bullet}KU^D)(\textbf{m})(G/H))} & {\pi_k((\text{GL}_1^{\bullet}KU^D)(\textbf{n})(G/H))} \\
	{\pi_k((\Omega^{m}KU^D_m)^H)} & {\pi_k((\Omega^{n}KU^D_n)^H)}
	\arrow["\cong", from=1-2, to=2-2]
	\arrow["\cong"', from=1-1, to=2-1]
	\arrow[from=1-1, to=1-2]
	\arrow[from=2-1, to=2-2]
\end{tikzcd}\]
where the bottom horizontal arrow is an isomorphism when $m > 0$.
\end{proof}

\section{Automorphisms of $D \otimes \bK$}

\subsection{The $\cI$-$G$-space $\cG_D^{\bullet}$}
\label{G_D}

The $C^*$-algebra $D=\Endo{V}^{\otimes \infty}$ is unital, separable and strongly self-absorbing. It has been shown in \cite{paper:DP-Dixmier} that the automorphism group of the stabilisation of such algebras carries an infinite loop space structure.
In order to extend this to the equivariant case we consider $D \otimes \bK$ as a $G$-$C^*$-algebra for a finite group $G$ (which we will later choose to be $\Z/n\Z$) with the action on $D$ defined in \eqref{eqn:D}. 

Let $\cG_D^{\bullet}(\textbf{m})=\Aut{(D \otimes \bK)^{\otimes m}}$ (with $(D \otimes \bK)^{\otimes 0} :=\C$, so that $\cG_D^{\bullet}(\textbf{0})=*$).
The functor $\cG_D^{\bullet}$ is defined on a morphism $\alpha \colon \textbf{m} \rightarrow \textbf{n}$ in~$\cI$ as \[\cG_D^{\bullet}(\alpha)(g)=\bar{\alpha} \circ (\id{(D \otimes \bK)^{\otimes (n-m)}} \otimes g) \circ \bar{\alpha}^{-1},\]
with $g \in \cG_D^{\bullet}(\textbf{m})$ and $\bar{\alpha}$ as explained at the beginning of Section \ref{equnits}. The tensor product equips $\cG_D^{\bullet}(\textbf{n})$ with the structure of a commutative $\cI$-$G$-monoid:
\[ \mu_{m,n}: \cG_D^{\bullet}(\textbf{m}) \times \cG_D^{\bullet}(\textbf{n}) \rightarrow \cG_D^{\bullet}(\textbf{m} \sqcup \textbf{n})\ , \qquad (g,h) \mapsto g \otimes h\ .\] 
Note that in this case $\cG_D^{\bullet}(\textbf{n})$ has both a levelwise group multiplication and a monoid structure and they satisfy the Eckmann-Hilton condition from \cite[Def.~3.1]{paper:DP-Units}. We will also need a compatibility condition for inverses:

\begin{definition}
    Let $X$ be an $\cI$-$G$-monoid such that $X(\textbf{n})$ is also a group for each \textbf{n}. Let $\iota_m:\textbf{0} \rightarrow \textbf{m}$ be the unique such morphism in $\cI$. We say that $X$ has \textit{compatible inverses} if there is a path from $(\iota_m \sqcup \id{\textbf{m}})_*(g) \in X(\textbf{m} \sqcup \textbf{m})$ to $(\id{\textbf{m}} \sqcup \iota_m)_*(g)$ for all $m$ and $g \in X(\textbf{m})$.
\end{definition}

If $X$ satisfies the Eckmann-Hilton condition \cite[Def.~3.1]{paper:DP-Units} and has compatible inverses, then for each $g \in X(\textbf{m})$ there is a path connecting $\mu_{m,m}(g,g^{-1}) \in X(\textbf{m} \sqcup \textbf{m})$ and $1_{m \sqcup m} \in X(\textbf{m} \sqcup \textbf{m})$ \cite[Lem. 3.3]{paper:DP-Units}. 

From now on we restrict to $G = \Z/n\Z$ for $n \in \N$.
\begin{lemma}
\label{GD_stable}
    $\cG_D^{\bullet}$ is a stable commutative $\cI$-$\Z/n\Z$-monoid with compatible inverses.
\end{lemma}

\begin{proof}
    See \cite[Thm. 4.5]{paper:DP-Units} for $H=\{e\}$. The proof for $H = \Z/k\Z \subset \Z/n\Z$ is an easy adaptation of \cite[Lem. 4.2]{paper:EP-Circle} where this is proven for $H = \bT$. 
\end{proof}

Recall that any $\cI$-$G$-monoid can be regarded as an $\cI$-$\cO_G$-monoid by passage to fixed points, which preserves monoidality. By a slight abuse of notation, we shall refer to $\cG_D^{\bullet}$ as an $\cI$-$\cO_G$-monoid in Section \ref{final_2}, where we will compare it with the $\cI$-$\cO_G$-space of units.

\subsection{$\Z/n\Z$-equivariant homotopy type of $\Aut{D \otimes \bK}$}
\label{Aut}

Let $G = \Z/n\Z$. Since the $G$-$C^*$-algebra $D$ defined in \eqref{eqn:D} is an infinite tensor product of a single representation, one can compute its equivariant $K$-theory using continuity. Note that the map $T \mapsto T \otimes 1_D$ induces multiplication by $V$ on $K$-theory. Therefore,
\[
    K_0^{G}(\C) \cong K_0^G(\Endo{V}^{\otimes k}) \cong \Z[t]/(t^n - 1) =: R \quad \text{and} \quad K_0^{G}(D) \cong R[p_V^{-1}]\ ,
\]
where $p_V \in R$ is the polynomial corresponding to the $G$-representation $V$. The $\Z$-basis $1, t, \dots, t^{n-1}$ of $R$ corresponds to the irreducible representations of $\Z/n\Z$. The group $K_0^G(D)$ is in fact a ring. But it also has the structure of an ordered $K_0^G(\C)$-module. The positive elements in $K_0^G(\C)$ are given by elements represented by representations (as opposed to their formal differences). Under the isomorphism $K_0^G(\C) \cong R$ the positive elements are those polynomials with non-negative coefficients. Since $K_0^G(D)$ is an ordered colimit of these groups, its positive cone corresponds under the isomorphism with $R[p_V^{-1}]$ to $R[p_V^{-1}]_+$, which is defined for $q \in R$ and $k \in \N_0$ by
\[
    \frac{q}{p_V^k} \in R[p_V^{-1}]_+ \  \Leftrightarrow \  q\cdot p_V^l \text{ has non-negative coefficients for some } l \in \N_0\ .
\]
For details about this type of action and the order structure we refer the reader to \cite{memoir:Handelman, paper:HandelmanRossmann}. We need to consider the subgroup of positive units: Let $R_+ \subset R$ be the semiring of positive elements contained in a commutative unital ring $R$ with an ordered group structure and define
\[
    GL_1(R)_+= \{ x \in GL_1(R) \cap R_+ \ | \ x^{-1} \in R_+ \}\ ,
\]
i.e.\ those elements of $R_+$ which are invertible as elements of $R$ and whose inverse is in $R_+$ (see Remark~\ref{rem:glitch} for the reason why this definition is different from the one in \cite[p.~17]{paper:EP-Circle}). 

For a $G$-$C^*$-algebra $B$ denote by $\AutSt{G}{B}$ the $G$-equivariant automorphisms (which coincide with the fixed points with respect to the conjugation action of $G$ on $\Aut{B}$). For a projection $q \in B$ let $\AutSt{G, q}{B}$ be the stabiliser subgroup of $q$. Let $\Proj{q}{B}$ be the path-component of $q$ in the subspace of projections in $B$ and let $\AutSt{G,0}{B}$ be the path-component of $\id{B}$ in $\AutSt{G}{B}$.

In the following lemma we collect some results that are easily adapted from the case $G=\bT$ discussed in \cite{paper:EP-Circle} to the subgroups $\Z/k\Z \leq Z/n\Z \le \bT$. 

\begin{lemma} \label{lem:hom_equivalences}
    Let $G = \Z/n\Z$ and let $e \in \bK$ be a projection of rank one whose image is the one-dimensional trivial $G$-representation. 
    \begin{enumerate}[(i)]
        \item \label{it:Aut_contractible} The topological groups $\Aut{D}$ and $\AutSt{1_D\otimes e}{D \otimes \bK}$ are both $G$-con\-trac\-tible. In particular, for any subgroup $H \le \Z/n\Z$ the subgroups $\eqAut{H}{D}$ and $\AutSt{H, 1_D\otimes e}{D \otimes \bK}$ are both contractible.
        \item \label{it:ProjBU} There is a homotopy equivalence
                \[ 
                    \Proj{1_D \otimes e}{(D \otimes \bK)^H} \simeq BU(D^H)
                \]
                for $H \le \Z/n\Z$.
        \item \label{it:AutProj} There is a homotopy equivalence
                \[
                    \eqAutId{H}{D \otimes \bK} \simeq \Proj{1_D \otimes e}{(D \otimes \bK)^H}
                \]
                for $H \le \Z/n\Z$.
    \end{enumerate}
\end{lemma}

\begin{proof}
Write $A = D \otimes \bK$ and $q_0 = 1_D \otimes e$. Part (\ref{it:Aut_contractible}) uses the contractions constructed in \cite[Thm.~2.6]{paper:EP-Circle}. We discuss the case $\AutSt{q_0}{A}$; the one for $\Aut{D}$ is very similar. Carrying \cite[eqn.~(9)]{paper:EP-Circle} over from $\bT$ to the subgroup $H \le \Z/n\Z \le \bT$ gives a $G$-equivariant continuous path
    \[
        \gamma \colon [0,1] \to \hom_H(A, A \otimes A)
    \]
    with $\gamma(0)(a) = a \otimes q_0$, $\gamma(1)(a) = q_0 \otimes a$ and running through isomorphisms for $t \in (0,1)$. Moreover, $\gamma(t)(q_0) = q_0 \otimes q_0$ for all $t$. Define
    \[
        \Theta(\beta, t) =
        \begin{cases}
            \beta & t = 0\ ,\\
            \gamma(t)^{-1} \circ (\beta \otimes \id{A}) \circ \gamma(t) & 0 < t < 1\ ,\\
            \id{A} & t = 1\ .
        \end{cases}
    \]
    Since $\beta(q_0) = q_0$ we have $(\beta \otimes \id{A})(q_0 \otimes q_0) = q_0 \otimes q_0$, so $\Theta(\beta,t)$ again fixes $q_0$, and clearly $\Theta(\id{A},t) = \id{A}$. The path $\gamma$ is equivariant in the sense that $\gamma(t) \circ \alpha_g = (\alpha_g \otimes \alpha_g) \circ \gamma(t)$, which implies $\Theta(\alpha_g \circ \beta \circ \alpha_g^{-1}, t) = \alpha_g \circ \Theta(\beta,t) \circ \alpha_g^{-1}$. Therefore the homotopy is $G$-equivariant.

    For part (\ref{it:ProjBU}) it suffices to see that $\Proj{q_0}{A^H}$ fits into a fibration sequence
    \[
        U(D^H) \to U(M(A^H)) \to \Proj{q_0}{A^H}
    \]
    constructed as in \cite[Prop.~3.2]{paper:EP-Circle}, where $U(M(A^H))$ is the unitary group of the multiplier algebra of $A^H$ equipped with the norm topology. Since $A^H$ is stable, point (\ref{it:ProjBU}) follows from the contractibility of $U(M(A^H))$.

    Part (\ref{it:AutProj}) uses a similar fibre bundle
    \[
        \AutSt{H,q_0}{A} \to \eqAutId{H}{A} \to \Proj{q_0}{A^H}\ ,
    \]
    which is constructed as in \cite[Prop.~3.3]{paper:EP-Circle} together with the contractibility of the fibre proven in (\ref{it:Aut_contractible}). 
\end{proof}

Lemma \ref{lem:hom_equivalences} allows us to determine the equivariant homotopy type of $\Aut{D \otimes \bK}$ using the evaluation map
\begin{equation} \label{eqn:Aut_to_proj}
    \eqAut{H}{D \otimes \bK} \to \Proj{}{(D \otimes \bK)^H} \quad, \quad \beta \mapsto \beta(1_D \otimes e)
\end{equation}
suitably restricted to the correct path-components in the codomain for a subgroup $H \le \Z/n\Z$. As in \cite[Lem.~2.8]{paper:EP-Circle} Rieffel's criterion applies to the $\Z/n\Z$-action on $\bK$ and since colimits of saturated actions are saturated, it follows that the inclusion map induces an isomorphism
\[
    K_0((D \otimes \bK)^H) \to K_0^{H}(D \otimes \bK) \cong K_0^{H}(D) \ .
\]
Combining \eqref{eqn:Aut_to_proj} with the natural map $\pi_0(\Proj{}{(D \otimes \bK)^H}) \to K_0((D \otimes \bK)^H)$ and the above isomorphisms therefore gives a map
\[
   \kappa \colon \pi_0(\eqAut{H}{D \otimes \bK}) \to K_0^{H}(D)\ .
\]

\begin{lemma}
\label{units+}
For every subgroup $H = \Z/m\Z \leq G = \Z/n\Z$ the map $\kappa$ restricts to a group isomorphism
\[ \pi_{0}(\eqAut{H}{D \otimes \bK}) \xrightarrow{\cong} GL_1(K_{0}^H(D))_+\ .\]
\end{lemma}

\begin{proof}
The projection $1_D \otimes e$ represents $1 \in K_0^H(D)$. Since $K$-theory is homotopy-invariant we have a homomorphism 
\[
    \pi_0(\eqAut{H}{D \otimes \bK}) \to \End(K_0^H(D)) \quad , \quad [\alpha] \mapsto \alpha_*\ .
\]
The induced homomorphism $\alpha_*$ is $K_0^H(\C)$-linear. Recall that 
\[
    K_0^H(\C) \cong \Z[t]/(t^m-1) = R \quad, \quad K_0^H(D) \cong R[p_V^{-1}]\ .
\]
The element corresponding to $p_V$ in $K_0^H(\C)$ acts invertibly on $K_0^H(D)$. This implies that the above homomorphism is actually $K_0^H(D)$-linear, hence given by multiplication by $\alpha_*(1) = \kappa(\alpha)$ with inverse given by $(\alpha^{-1})_*(1)$. As a consequence, $\kappa$ gives a well-defined group homomorphism to $GL_1(K_0^H(D))_+$.

The argument for injectivity is the same as in the proof of \cite[Lem.~2.9]{paper:EP-Circle}: Let $\alpha \in \eqAut{H}{D \otimes \bK}$ such that $\kappa([\alpha]) = 1 \in K_0^H(D)$. Let $q_0 = \alpha(1_D \otimes e) \in (D \otimes \bK)^H$. Since $(D \otimes \bK)^H$ is an AF-algebra, it has the cancellation property. Hence, there is a unitary $u$ in the unitary group of the unitisation of $(D \otimes \bK)^H$ such that $u\alpha(q_0)u^* = q_0$. Since the unitary group of a unital AF-algebra is path-connected, we can choose a path $u_t$ from $1$ to $u$ and $t \mapsto \Ad_{u_t} \circ \alpha$ is a path in $\eqAut{H}{D \otimes \bK}$ from $\alpha$ to an element of $\AutSt{H,q_0}{D \otimes \bK}$. But this group is contractible by Lemma~\ref{lem:hom_equivalences} (\ref{it:Aut_contractible}), so in particular path-connected.

To see surjectivity let $q \in GL_1(R[p_V^{-1}])_+$. Then
\[
    q = \frac{r}{p_V^k}
\]
for some $r \in R_+$, $k \in \N_0$ and by our definition of $GL_1(R[p_V^{-1}])_+$ there is $s \in R_+$ such that 
\(
    r \cdot s = p_V^l
\) for some $l \in \N_0$. Since $r$ and $s$ are both positive, this corresponds to a tensor product decomposition of the $H$-representation $V^{\otimes l}$, i.e.\ we have two representations $V_r$ and $V_s$ corresponding to $r$ and $s$, respectively, such that
\begin{equation} \label{eqn:V_End_iso}
    V_r \otimes V_s \cong V^{\otimes l} \quad \text{and} \quad \Endo{V_r} \otimes \Endo{V_s} \cong \Endo{V}^{\otimes l}\ .
\end{equation}
Observe that because of \eqref{eqn:V_End_iso} the $H$-$C^*$-algebra~$D$ tensorially absorbs $\Endo{V^r}$. Let  
\[
    \psi_r \colon \Endo{V_r} \otimes D \cong \Endo{V_r} \otimes (\Endo{V_r} \otimes \Endo{V_s})^{\otimes \infty} \to D
\]
be the $H$-equivariant $*$-isomorphism that shifts the tensor factors and identifies the result with $D$ using \eqref{eqn:V_End_iso}.
Since $V_r \otimes \mathcal{H} \cong \mathcal{H}$ for $\mathcal{H} = \ell^2(G) \otimes \mathcal{H}_0$, we can also choose an $H$-equivariant $*$-isomorphism 
\[
    \phi_r \colon \bK(\mathcal{H}) \otimes \Endo{V_r} \to \bK(\mathcal{H} \otimes V_r) \to \bK(\mathcal{H})\ .
\]
The composition $(\phi_r \otimes \id{D}) \circ (\id{\bK} \otimes \psi_r^{-1})$ gives $\alpha_r \in \eqAut{H}{D \otimes \bK}$ after swapping $D$ and $\bK$. By construction, $\alpha_r(1_D \otimes e)$ is an invariant projection whose image is the Hilbert $D$-module $D \otimes V_r$, i.e.\ $[\alpha_r(1_D \otimes e)] = r \in R_+$. Similarly, we can construct $\alpha_{p_V} \in \eqAut{H}{D \otimes \bK}$ with $[\alpha_{p_V}(1_D \otimes e)] = p_V$. Since $\kappa$ is a group homomorphism, we have $\kappa([(\alpha_{p_V}^{-1})^k \circ \alpha_r]) = q$. 
%If $V_s$ does not contain the trivial representation, then there is $m \in \Z/p\Z$ such that $\C_m \otimes V_s$ does, where $\C_m$ is the character corresponding to $m$. The above argument then shows that $t^{-m}q(t)$ is represented by an automorphism $\beta$. As described in the proof of \cite[Lem.~2.9]{paper:EP-Circle} there is an equivariant automorphism $\alpha_t \in \eqAut{\Z/p\Z}{\bK}$ with the property $[\alpha_t(e)] = t \in R_+$ and we have $\kappa([(\id{D} \otimes \alpha_t^m) \circ \beta]) = q$.
\end{proof}

\begin{remark} \label{rem:glitch}
    There is a small mistake in the proof of \cite[Lem.~2.9]{paper:EP-Circle}: The definition of $GL_1(R_+)$, the surjectivity part of the proof of the lemma and the computation of the generators of $GL_1(R_+)$ are actually only correct under the additional assumption that the prime factors of $p_V$ are positive. This is not automatic as the example
    \[
        p_V(t) = 1+t^3 = (1 - t + t^2)(1+t)
    \]
    shows. The above proof of surjectivity also works in the case $G = \bT$ and fixes this gap. 
\end{remark}

\begin{definition} \label{def:projectively_faithful}
    Let $G$ be a finite group and let $\rho \colon G \to U(V)$ be a unitary representation on a finite-dimensional complex inner product space. Let $q \colon U(V) \to PU(V) = U(V)/U(1)$ be the quotient map. We say that $V$ is \emph{projectively faithful} if $\dim(V) \geq 2$ and $q \circ \rho \colon G \to PU(V)$ is injective.
\end{definition}

The following elementary consideration will be needed in order to compute the $K_0$-group of the fixed point algebra of $D$.

\begin{lemma}
\label{coefficients}
Let $\rho \colon \Z/n\Z \to U(V)$ be a projectively faithful representation, and let $p_{V}=\sum_{i=0}^{n-1} a_{i} t^{i}$ be the associated character polynomial. Then there exists $N \in \N$ such that the $N$-th power
\[p_{V}^{N}=\sum_{j=0}^{n-1} b_{j} t^{j}\]
has coefficients $b_{j} \ge 2$ for all $j \in \{0,\dots,n-1\}$.
\end{lemma}

\begin{proof}
Let $\chi_V(g) = \text{tr}(\rho(g))$ be the character of $\rho$ and $d$ its dimension. Since $\rho$ is unitarily equivalent to a direct sum $\mu_1 \oplus \dots \oplus \mu_d$ of one-dimensional representations $\mu_i$ we have 
\[
    \lvert \chi_V(g)\rvert = \lvert \mu_1(g) + \dots + \mu_d(g) \rvert \leq d\ .
\]
Each $\mu_i(g)$ has modulus $1$, therefore equality can only hold if $\mu_1(g) = \dots = \mu_d(g)$. Thus, since $q \circ \rho$ is injective, we must have 
\[
    \lambda = \max_{\substack{g \in G \\ g \neq e}} \lvert \chi_V(g) \rvert < d\ .
\]
Let $t^j$ be the irreducible representation $k \mapsto e^{\frac{2\pi i j k}{n}}$. We can compute the coefficients of $p_V^N$ using the inner product of characters and the orthogonality of the irreducibles. Hence
\[
    b_j = \langle\chi_V^N, t^j \rangle = \frac{1}{n} \sum_{g \in G} \chi_V(g)^N\,\overline{t^j(g)} = \frac{d^N}{n} + \frac{1}{n} \sum_{\substack{g \in G \\ g \neq e}} \chi_V(g)^N\,\overline{t^j(g)}\ ,
\]
which leads to the estimate
\[
    \frac{d^N}{n} - b_j \leq 
    \left\lvert \frac{d^N}{n} - b_j \right\rvert \leq \frac{1}{n} \sum_{\substack{g \in G \\ g \neq e}} \lvert \chi_V(g)^N\rvert \leq \frac{(n-1)\lambda^N}{n} \leq \lambda^N
\]
giving $b_j \ge \frac{d^N}{n}\left(1 - n \left(\frac{\lambda}{d}\right)^N\right)$, which tends to infinity for large $N$.
\end{proof}

\begin{lemma}
\label{kt_fixed}
Let $\rho \colon \Z/n\Z \to U(V)$ be a projectively faithful representation and let $D$ be its associated UHF-algebra. Let $H \le \Z/n\Z$. The $K$-groups of the fixed-point algebra $D^{H}$ are given by
\[K_0(D^{H}) \cong R(H)[p_V^{-1}] \cong K_0^{H}(D) \text{ and } K_1(D^{H})=0\ .\]
\end{lemma}

\begin{proof}
%The proof works similarly as in \cite[Lem. 3.5]{paper:EP-Circle}. 
First note that because the group acts on each factor separately and we have conditional expectations onto the $H$-invariant subalgebras, the fixed-point algebra $D^{H}$ is isomorphic to the colimit
\[ \End(V)^{H} \rightarrow \cdots \rightarrow \bigl(\End(V)^{\otimes l}\bigr)^{H} \rightarrow \bigl(\End(V)^{\otimes (l+1)}\bigr)^{H} \rightarrow \cdots,\]
where the connecting $*$-homomorphisms are given by $T \mapsto T \otimes 1$. Let $z \in \bT$ be a generator of $H$ when viewed as a subgroup of the circle. For $k \in \{0,\dots,\lvert H \rvert -1\}$ let 
\[
    V_k^l = \{ v \in V^{\otimes l} \ | \ \rho(z)^{\otimes l}v = z^kv \} \quad, \quad S^l = \{ k \in \{0, \dots, \lvert H \rvert - 1\}\ | \ V_k^l \neq 0\}
\]
be the isotypic components of $V^{\otimes l}$ and their support in the set of characters. By Schur's lemma we have 
\[
    \End(V^{\otimes l})^H  = \End_H(V^{\otimes l}) \cong \bigoplus_{k \in S^l} \End(V_k^l)\ .
\]
Note that the elements of $S^l$ correspond exactly to the indices of non-zero coefficients of the character polynomial $(p_V)^l$ of $V^{\otimes l}$. We have 
\[
    K_0((\End(V)^{\otimes l})^H) \cong R_{S^l}\ ,
\]
where $R_{S^l} \subset R(H) \cong \Z[t]/(t^{\lvert H \vert} - 1)$ is the subgroup generated by the irreducibles (powers of $t$) labelled by $S^l$. Moreover, 
\begin{equation} \label{eqn:colimit-K0}
    K_0((\End(V)^{\otimes l})^H) \to K_0^H(\End(V)^{\otimes l})
\end{equation}
corresponds to the embedding $R_{S^l} \to R(H)$. If we apply $K_0$ to each term in the colimit and identify a representation with its character polynomial, the connecting homomorphism turns into multiplication by $p_V$. Lemma~\ref{coefficients} implies that $(p_V)^l$ contains all powers of $t$ with positive coefficients for sufficiently large $l$. Hence, eventually $R_{S^l} = R(H)$. The claim for $K_0$ follows by combining these two considerations and continuity of $K$-theory. Since $D^{H}$ is an AF-algebra we also have $K_1(D^{H})=0$.

Chasing the homomorphisms \eqref{eqn:colimit-K0} through the colimit, noting that the connecting homomorphisms correspond to multiplication by $p_V$ we see that 
\[
    K_0(D^H) \to K_0^H(D) \cong R(H)[p_V^{-1}]
\]
is an isomorphism. 
\end{proof}

We can now compute the homotopy groups of $\eqAut{H}{D \otimes \bK}$ for subgroups $H \leq \Z/n\Z$. By Lemma~\ref{lem:hom_equivalences} (\ref{it:ProjBU}) and (\ref{it:AutProj}), this boils down to computing the ones of $U(D^{H})$. Here the fact that the representation ring of $H$ has finite rank over $\Z$ plays a fundamental role.

\begin{theorem}
\label{homotopyU}
Let $V$, $D$ and $H$ be as in Lemma~\ref{kt_fixed}. Let $p_V$ be the character polynomial of $V$. We have $\pi_{2k}(U(D^{H}))=0$ and
\[\pi_{2k+1}(U(D^{H})) \cong R(H)[p_{V}^{-1}] \cong K_{0}^{H}(D)\]
for $k \in \N_{0}$.
\end{theorem}

\begin{proof}
Let $m \in \N_{0}$. The group $\pi_m(U(D^H))$ is the colimit of the sequence
\[
\cdots \rightarrow \pi_m(U(\End(V^{\otimes l})^{H})) \rightarrow \pi_m(U(\End(V^{\otimes (l+1)})^{H})) \rightarrow \cdots
\]
where the connecting homomorphism is induced by the map $T \mapsto T \otimes 1$. Using the notation from Lemma~\ref{kt_fixed} let $V^l_k$ for $k \in \{0, \dots, \lvert H \rvert -1\}$ be the isotypical components of $V^{\otimes l}$ and let $S^l$ be their support. We have 
\[
    \End(V^{\otimes l})^H \cong \bigoplus_{k \in S^l} \End(V_k^l)\ .
\]
Denote the coefficients of $(p_V)^l$ by $b_0, \dots ,b_{\lvert H \rvert -1}$ and note that $b_k = \dim(V_k^l)$ (in particular $b_k = 0$ if $k \notin S^l$).  Hence,
\[
    U(\End(V^{\otimes l})^H) \cong \prod_{k \in S^l} U(b_k)\ .
\]
By Lemma \ref{coefficients} there exists $N \in \N$ such that $(p_{V})^{N}$ has all coefficients $\ge 2$. Since the sequence
\[
\End(V^{\otimes 0})^H \rightarrow \End(V^{\otimes N})^H \rightarrow \cdots \rightarrow \End(V^{\otimes mN})^H \rightarrow \cdots
\]
has the same colimit $D^H$, we can work with $V^{\otimes N}$ instead of $V$ and assume without loss of generality that $p_{V}$ only has coefficients $\ge 2$.
With our new choice of $p_V$ the coefficients $b_{j}$ of $(p_{V})^{l}$ then satisfy $b_{j} > l$ for $l \in \N$. This implies that $b_j> \frac{m}{2}$ whenever $l > \frac{m}{2}$. Now let $U_{\infty}$ be the colimit over the inclusions $U(m) \hookrightarrow U(m+1)$ that add a $1$ in the lower right corner. In the above situation $m$ falls into the stable range for the homotopy groups, which implies that the inclusion $U(b_j) \to U_\infty$ induces an isomorphism $\pi_{m}(U(b_{j})) \cong \pi_{m}(U_{\infty})$, which is isomorphic to $\Z$ if $m$ is odd and vanishes if $m$ is even. 

Therefore for $l > \frac{m}{2}$, each term in the sequence which computes the group $\pi_{m}(U(D^H))$ for odd $m$ is isomorphic to the direct sum of one copy of $\Z$ for every irreducible $H$-representation, i.e., is isomorphic to the representation ring $R(H)$ as $R(H)$-modules. Since the connecting homomorphism is given by multiplication by $p_{V}$, the result follows after taking the colimit. 
\end{proof}

\begin{remark}
\label{K1}
The proof of Theorem \ref{homotopyU} shows in particular that the natural map
\[\pi_{m}(U(D^H)) \rightarrow K_1(C(S^m,D^H))\]
given by mapping the class of $\gamma: S^m \rightarrow U(D^H)$ to the class of the corresponding unitary $u_{\gamma} \in U(C(S^m,D^H))$ is an isomorphism. 
\end{remark}

\begin{corollary} \label{cor:pikAut}
    Let $H \le \Z/n\Z$. Let $D$ be the UHF-algebra associated to a projectively faithful $\Z/n\Z$-representation $V$. We have 
    \[
        \pi_m(\eqAut{H}{D \otimes \bK}) \cong \begin{cases}
            GL_1(R(H)[p_V^{-1}])_+ & \text{if } m = 0, \\
            R(H)[p_V^{-1}] &\text{if } m > 0 \text{ and } m \text{ is even},\\
            0 & \text{if } m \text{ is odd}. 
        \end{cases}
    \]
\end{corollary}

\begin{proof}
This follows by combining Lemma~\ref{units+}, Lemma~\ref{lem:hom_equivalences}~(\ref{it:ProjBU}) and (\ref{it:AutProj}), and Theorem~\ref{homotopyU}.
\end{proof}

\section{$\Z/n\Z$-equivariant $\Aut{D \otimes \bK}$-bundles and units of $KU^D$}
\label{final_section}

\subsection{The classifying space of equivariant $\Aut{D \otimes \bK}$-bundles}
\label{final_1}
In this section we will take a closer look at equivariant bundles of $C^*$-algebras with fibre given by the $G$-$C^*$-algebra $D \otimes \bK$, where $D$ is the UHF-algebra associated to a projectively faithful representation. We will assume that they are locally trivial in the following sense:
\begin{definition} \label{def:eq_bundles}
    Let $G$ be a finite group. Let $A$ be a $G$-$C^*$-algebra and denote the action by $\alpha$. Let $X$ be a topological $G$-space and let $\pi \colon \cA \to X$ be a locally trivial $C^*$-algebra bundle with fibre $A$ over $X$. We will say that $\cA$ is a \emph{locally trivial $(G,A,\alpha)$-bundle} (or a \emph{locally trivial $(G,A)$-bundle} if the action is clear) if 
    \begin{enumerate}[a)]
        \item $G$ acts from the left on the total space $\cA$ in such a way that $\pi$ is $G$-equivariant,
        \item for every $x \in X$ there exists a $G$-invariant open neighbourhood $U \subseteq X$ and a $G$-equivariant homeomorphism 
        \[
            \varphi_U \colon U \times A \to \left.\cA\right|_U\ ,
        \]
        where $\left.\cA\right|_U = \pi^{-1}(U)$ and $G$ acts on $U \times A$ via $g \cdot (x,a) = (g \cdot x, \alpha_g(a))$.
    \end{enumerate}
\end{definition}

To save space we will drop the adjective ``locally trivial'' from the notation in the following and instead mention the cases explicitly where we do not assume it. By elementary bundle theory each $C^*$-algebra bundle $\cA \to X$ with fibre $A$ gives rise to a principal $\Aut{A}$-bundle $\cP \to X$ such that there is a bundle isomorphism
\(
    \cA \cong \cP \times_{\Aut{A}} A
\). In fact, this construction induces a bijection between isomorphism classes of $C^*$-algebra bundles with fibre $A$ and principal $\Aut{A}$-bundles. 

We outline a few details of the construction of $\cP$: Let $\cA_x = \pi^{-1}(\{x\})$. The fibre $\cP_x$ of $\cP$ over $x \in X$ is given by
\begin{equation} \label{eqn:Px}
    \cP_x = \text{Iso}(A, \cA_x)\ ,
\end{equation}
which has a canonical right $\Aut{A}$-action, and the topology on $\cP$ is fixed by the property that a local trivialisation $\varphi_U \colon U \times A \to \left.\cA\right|_U$ over $U$ gives rise to a homeomorphism
\[
    U \times \Aut{A} \to \left.\cP\right|_U  \ .
\]
Now assume that $\cA \to X$ is a $(G,A)$-bundle. We will see that the associated principal $\Aut{A}$-bundle is equivariant in the following sense:
\begin{definition} \label{def:GAut-bundle}
    Let $G$ be a finite group, let $X$ be a topological $G$-space and let $A$ be a $G$-$C^*$-algebra. A $(G,\Aut{A})$-bundle is a locally trivial principal $\Aut{A}$-bundle $\cP \to X$ together with a left $G$-action by bundle maps such that each $x \in X$ has a $G$-invariant trivialising neighbourhood $U \subseteq X$ and a $G$-equivariant trivialisation
    \[
        \left.\cP\right|_U \to U \times \Aut{A}
    \]
    where $G$ acts from the left on $\Aut{A}$ via the action on $A$ and diagonally on $U \times \Aut{A}$.
\end{definition}

%Note that $(G,\Aut{A}$)-bundles are locally trivial equivariant bundles in the sense of \cite[Def.~2.3]{paper:tomDieck} for the local object $\Aut{A} \to \ast$.

\begin{lemma}
    Let $G$ be a finite group and let $A$ be a $G$-$C^*$-algebra. The principal $\Aut{A}$-bundle associated to a $(G,A)$-bundle is a $(G,\Aut{A})$-bundle. Conversely, given a $(G,\Aut{A})$-bundle $\cP \to X$, then the associated $C^*$-algebra bundle is a $(G,A)$-bundle over $X$. 
\end{lemma}

\begin{proof}
The action of $G$ on $\cA$ induces a corresponding action on $\cP$ by post-composition (see \eqref{eqn:Px}) such that $G$ acts by bundle maps. Fix $x \in X$. Choose a $G$-invariant trivialising neighbourhood $U \subseteq X$ and a $G$-equivariant homeomorphism 
\[
    \varphi_U \colon U \times A \to \left.\cA\right|_U\ .
\]
The map $\varphi_U$ induces a $G$-equivariant trivialisation
\[
    \psi_U \colon U \times \Aut{A} \to \left.\cP\right|_U \quad , \quad (y, \beta) \mapsto \varphi_{U,y} \circ \beta\ .
\]
Conversely, assume that $\cP \to X$ is a $(G,\Aut{A})$-bundle. A $G$-equivariant trivialisation 
$\psi_U \colon U \times \Aut{A} \to \left.\cP\right|_U$ gives rise to a corresponding one for $\cA = \cP \times_{\Aut{A}} A$ defined by 
\[
    \varphi_U \colon U \times A \to \left.\cA\right|_U  \quad , \quad (y, a) \mapsto [\psi_U(y,\id{}),a] \ . \qedhere
\]
\end{proof} 

Following \cite{paper:tomDieck, book:MayClass} we will now construct a universal $(G,\Aut{A})$-bundle. Consider the following simplicial spaces: 
%$E\Aut{A} \to B\Aut{A}$ over the classifying space $B\Aut{A}$. Let
\[
    \cEG{n} = \Aut{A}^{n+1} \quad \text{ and } \quad \cBG{n} = \Aut{A}^n
\]
The face maps $d_i \colon \cEG{n} \to \cEG{n-1}$ are given by 
\[
    d_i(\beta_0, \dots, \beta_{n}) = (\beta_0, \dots, \beta_{i-1} \circ \beta_i, \dots,  \beta_{n}) \quad \text{ for } i > 0
\] 
and $d_0(\beta_0, \dots, \beta_{n}) = (\beta_1, \dots, \beta_{n})$. The degeneracy maps $s_i \colon \cEG{n} \to \cEG{n+1}$ are 
\[
    s_i(\beta_0, \dots, \beta_{n}) = (\beta_0, \dots, \beta_{i-1}, \id{}, \beta_i, \dots, \beta_{n})\ .
\]
The face and degeneracy maps for $\cBG{n}$ are defined similarly except for $d_n$, which drops the last component instead of composing. Both of these are simplicial $G$-spaces with respect to the action that is given by conjugation on the first $n$ factors of $\cEG{n}$ and $\cBG{n}$ and by the left action on the last factor of $\cEG{n}$. In the case of $\cEG{n}$ this is 
\[
    g \cdot (\beta_0, \dots, \beta_n) = (\alpha_g\beta_0\alpha_g^{-1}, \dots, \alpha_g\beta_{n-1}\alpha_g^{-1}, \alpha_g\beta_n)\ .
\]
Note that this action is compatible with the structure maps and that the obvious projection map $\cEG{\bullet} \to \cBG{\bullet}$ onto the first $n$ factors is $G$-equivariant. 

The two simplicial spaces $\cEG{\bullet}$ and $\cBG{\bullet}$ agree with the bar constructions $B(\ast, \Aut{A}_c, \Aut{A}_l)$ and $B(\ast, \Aut{A}_c, \ast)$ in the category of $G$-spaces, respectively, where $\Aut{A}_c$ is $\Aut{A}$ equipped with the conjugation action and $\Aut{A}_l$ is the same space equipped with the left action. Alternatively, the space $B\Aut{A}$ can be also viewed as the geometric realisation of the nerve of the topological category associated to the group $\Aut{A}$. Define 
\begin{equation} \label{eqn:EAutBAut}
    E\Aut{A} = \lvert \cEG{\bullet} \rvert \quad , \quad
    B\Aut{A} = \lvert \cBG{\bullet} \rvert\ .
\end{equation}
The projection map $\cEG{\bullet} \to \cBG{\bullet}$ induces a continuous $G$-equivariant map $\pi \colon E\Aut{A} \to B\Aut{A}$ on the geometric realisation.

Suppose that $\cA \to X$ is a $(G, A, \alpha)$-bundle with associated principal $\Aut{A}$-bundle $\cP \to X$. Let $x \in X$ and let $H \leq G$ be the stabiliser subgroup of $x$. We have for $h \in H$ and $p \in \cP_x = \text{Iso}(A, \cA_x)$ the homomorphism $\rho_p(h) = (p^{-1} \circ \left. h \right|_{\cA_x} \circ p) \in \Aut{A}$ such that 
\[
    h \cdot p =  p \cdot \rho_p(h)\ .
\]
Hence, after fixing $p$, we obtain an isomorphism over the orbit $G \cdot x$:
\[
    \left.\cP \right|_{G \cdot x} \cong G \times_{\rho_p} \Aut{A}
\]
as a bundle over $G\cdot x \cong G/H$. The right hand side is a $(G, \Aut{A})$-bundle over the $G$-space $G/H$. These are called local objects in \cite{paper:tomDieck}. Let
\begin{equation} \label{eqn:cRH}
    \cR_H = \{ \beta \circ \left.\alpha\right|_H \circ \beta^{-1} \colon H \to \Aut{A}\ |\ \beta \in \Aut{A} \}\ ,
\end{equation}
i.e.\ the family of all conjugates of $\left.\alpha\right|_H$. 

\begin{definition} \label{def:universal}
    Let $G$ be a finite group and let $A$ be a $G$-$C^*$-algebra. A principal $(G,\Aut{A})$-bundle $P \to B$ over the $G$-space $B$ is called universal if, for every $G$-CW-complex $X$, pulling back induces a bijection between the homotopy set $[X,B]_G$ and the set of isomorphism classes of principal $(G,\Aut{A})$-bundles over $X$. In this case we call $B$ an (equivariant) classifying space for $\Aut{A}$.
\end{definition}

\begin{prop} \label{prop:classifying_space}
    Let $G = \Z/n\Z$. Let $V$ be a finite-di\-men\-sional unitary $G$-representation. Let $D = \Endo{V}^{\otimes \infty}$ and $\bK = \bK(\ell^2(G) \otimes H_0)$. Then 
    \[
        \pi \colon E\Aut{D \otimes \bK} \to B\Aut{D \otimes \bK}
    \]
    is a universal $(G,\Aut{D \otimes \bK})$-bundle. 
\end{prop}

\begin{proof}
Let $A = D \otimes \bK$. The map $\{\id{}\} \to \Aut{A}$ is a $G$-equivariant cofibration by Lemma~\ref{lem:G-cofib}. By \cite[Thm.~8.2]{book:MayClass} the map $\pi \colon E\Aut{A} \to B\Aut{A}$ is thus a principal $\Aut{A}$-bundle. The construction of the local trivialisations in \cite[Thm.~8.2]{book:MayClass} is based on the observation that the $n$th simplicial filtration step included in the full space $E\Aut{A}$ is a cofibration (see also \cite[Thm.~7.6]{book:MayClass}). Since this is true $G$-equivariantly, the trivialisation described at the end of the proof of \cite[Thm.~8.2]{book:MayClass} is in fact $G$-equivariant in the sense of Definition ~\ref{def:GAut-bundle} and $E\Aut{A} \to B\Aut{A}$ is a $(G,\Aut{A})$-bundle. 

To see that it is universal we will follow \cite[Thm.~5.1]{paper:tomDieck}: We have to show that for every local object $X = G \times_\rho \Aut{A} \to G/H$ induced by an element $\rho \in \cR_H$ with $\cR_H$ as in \eqref{eqn:cRH}, the associated bundle 
\begin{equation} \label{eqn:EAutX}
    (E\Aut{A} \times X)/\Aut{A} \to G/H
\end{equation}
is $G$-shrinkable. The total space $E\Aut{A}$ is given by the geometric realisation of $B(\ast, \Aut{A}_c, \Aut{A}_l)$, where $c$ denotes the conjugation action and $l$ the left action. The bundle in \eqref{eqn:EAutX} is $G$-shrinkable if and only if its fibre over $eH$ is $H$-contractible. This fibre is homeomorphic to the geometric realisation of 
$B(\ast, \Aut{A}_c, \Aut{A}_c)$, i.e.\ the product with $X$ changes the action in the last coordinate to the conjugation action as well. 

As in the non-equivariant case \cite[Prop.~7.5]{book:MayClass}, this simplicial space is contractible via the extra degeneracy $s_{n+1}(\beta_0, \dots, \beta_n) = (\beta_0, \dots, \beta_n, \id{})$, which is $H$-equivariant since $\id{}$ is fixed by conjugation. Explicitly, in the homogeneous coordinates $g_k = \beta_k \circ \dots \circ \beta_n$ (with inverse $\beta_n = g_n$ and $\beta_k = g_k \circ g_{k+1}^{-1}$), on which $H$ acts diagonally by conjugation, the face maps delete entries, the degeneracy maps repeat them, and $s_{n+1}$ appends $\id{}$. Any $f \in \Delta[1]$ is given by a monotone map $f \colon [n] \to [1]$. Replacing $g_k$ by $\id{}$ whenever $f(k) = 1$ defines an $H$-equivariant simplicial homotopy 
\[
    B(\ast, \Aut{A}_c, \Aut{A}_c) \times \Delta[1] \to B(\ast, \Aut{A}_c, \Aut{A}_c)
\]
from the identity to the constant map with value $\id{}$. Since geometric realisation preserves finite products of simplicial spaces, it yields an $H$-equivariant contraction of the fibre. Hence the bundle \eqref{eqn:EAutX} is $G$-shrinkable.
\end{proof}

\begin{remark}
    %The proof of universality shows that $E\Aut{D \otimes \bK}$ is in fact a model for $E_\mathcal{F(R)}\Aut{D \otimes \bK}$ for the family $\mathcal{R}$ of local representations in the sense of \cite[Def.~3.4]{paper:LueckUribe} that consists of all conjugates of the action $\alpha \colon G \to \Aut{D \otimes \bK}$ and the trivial homomorphism $\{\id{}\} \to \Aut{D \otimes \bK}$ (see \cite[Thm.~11.5]{paper:LueckUribe}). Note, however, that 
    Note that the equivariant principal bundles considered in~\cite{paper:LueckUribe} are less restrictive than our definition. In particular, \cite[Def.~2.1]{paper:LueckUribe} only demands that local triviality holds non-equivariantly, whereas Definition~\ref{def:GAut-bundle} asks for $G$-equivariant local trivialisations (as in \cite{paper:tomDieck}). With the more flexible notion in \cite{paper:LueckUribe}, Condition~(H) (see \cite[Def.~6.1]{paper:LueckUribe}) is needed to ensure homotopy invariance in the sense of \cite[Thm.~10.1]{paper:LueckUribe}: if $p \colon P \to B \times [0,1]$ is a $G$-equivariant principal bundle, then it is isomorphic to $i_0^*P \times [0,1] \to B \times [0,1]$, where $i_0 \colon B = B \times \{0\} \to B \times [0,1]$ is the inclusion. It is not clear to us, if Condition~(H) holds in our case.
\end{remark}

\noindent Note that the $G$-space $\Aut{D \otimes \bK}$ has two equivariant deloopings: 
\begin{enumerate}
    \item The classifying space $B\Aut{D \otimes \bK}$ equipped with the $G$-action induced by the conjugation action of $G$ on $\Aut{D \otimes \bK}$ defined above. 
    \item The $\Gamma$-$G$-space delooping given by 
    \[
        B_\otimes\Aut{D\otimes \bK} = (\underline{E}H_\cI\cG_D)_1
    \]
    induced by the tensor product structure on automorphisms.
\end{enumerate}

\begin{lemma} \label{lem:comp_vs_tensor}
Let $G=\Z/n\Z$ and let $D$ and $\bK$ be the $G$-$C^*$-algebras defined above. Then there is a zig-zag of weak $G$-equivalences
\[
    \begin{tikzcd}
        B_\otimes\Aut{D\otimes \bK} \ar[r,"\simeq"] & 
        \Omega B_\otimes B\Aut{D \otimes \bK} &
        B\Aut{D \otimes \bK} \ar[l,"\simeq" above]
    \end{tikzcd}    
\]
As a consequence, $B_\otimes\Aut{D \otimes \bK}$ is a classifying space for $G$-equivariant $D \otimes \bK$-bundles over finite $G$-CW-complexes.
\end{lemma}

\begin{proof}
    We carry out the construction in the proof of \cite[Thm.~3.6]{paper:DP-Units} in the category of $G$-spaces. Note that $G$ acts on $\cG_D^\bullet(\mathbf{m}) = \Aut{(D \otimes \bK)^{\otimes m}}$ by conjugation and all structure maps of this commutative $\cI$-monoid are $G$-equivariant. Applying the bar construction with respect to composition levelwise therefore yields an $\cI$-$G$-space $B\cG_D^\bullet$ with 
    \[
        B\cG_D^\bullet(\mathbf{m}) = B\Aut{(D \otimes \bK)^{\otimes m}}
    \]
    By the interchanging law 
    \[
        (\beta \otimes \gamma) \circ (\beta' \otimes \gamma') = (\beta \circ \beta') \otimes (\gamma \circ \gamma')\ ,
    \]
    the tensor product induces simplicial maps on the bar constructions, which make $B\cG_D^\bullet$ a commutative $\cI$-monoid. The $\Gamma$-space delooping of $B\cG_D^\bullet$ is $B_\otimes B\Aut{D \otimes \bK}$. The two maps in the statement of this Lemma are those of \cite[Thm.~3.6]{paper:DP-Units}: the left one is $B_{\otimes}$ applied to the levelwise maps $\cG_D^\bullet(\mathbf{m}) \to \Omega B\cG_D^\bullet(\mathbf{m})$, and the right one is the structure map of the $\Gamma$-space $B_\otimes B\cG_D^\bullet$ in degree one. Both are natural, hence $G$-equivariant.

    All constructions above are built from products (in compactly generated Hausdorff spaces), coproducts, geometric realisations of simplicial spaces with $G$ acting through maps of simplicial spaces and loop spaces with $G$ acting trivially on the loop coordinate. Hence, we can carry out the construction of the two maps in the category of $G$-spaces. Let $H \le G$. Taking $H$-fixed points commutes with geometric realisation by \cite[Lem.~12.3]{paper:MayMerlingOsorno}. Moreover, 
    \[
        (\Aut{(D \otimes \bK)^{\otimes m}})^H = \eqAut{H}{(D \otimes \bK)^{\otimes m}}\ .
    \]
    Hence, the $H$-fixed points of the zig-zag are the zig-zag
    \[
    \begin{tikzcd}
        B_\otimes\eqAut{H}{D\otimes \bK} \ar[r] & 
        \Omega B_\otimes B\eqAut{H}{D \otimes \bK} &
        B\eqAut{H}{D \otimes \bK} \ar[l]
    \end{tikzcd}    
    \]
    associated with the commutative $\cI$-monoid $(\cG_D^\bullet)^H$ of $H$-equivariant automorphisms. This $\cI$-monoid satisfies the hypotheses of the original~\cite[Thm.~3.6]{paper:DP-Units}: it is stable, commutative and has compatible inverses by Lem\-ma~\ref{GD_stable}. Each group $\AutSt{H}{(D \otimes \bK)^{\otimes m}}$ has the homotopy type of a CW-complex as the $H$-fixed points of a space which has the $G$-homotopy type of a $G$-CW-complex by Theorem~\ref{lem:GCW-complex}. Each group is also well-pointed by Lemma~\ref{lem:G-cofib}, since the $H$-fixed points of a $G$-cofibration form a cofibration. Hence, both maps in the zig-zag are weak $G$-equivalences. 

    For the last statement of the Lemma, recall that for a $G$-CW-complex $X$ the functor $[X, -]_G$ takes weak $G$-equivalences to bijections. Hence, 
    \[
        [X, B_\otimes \Aut{D \otimes \bK}]_G \cong [X, B\Aut{D \otimes \bK}]_G
    \]
    and the claim follows from Proposition~\ref{prop:classifying_space}.
\end{proof}

\subsection{Comparison between $\cG_D^{\bullet}$ and $\textnormal{GL}^{\bullet}_1 KU^D$}
\label{final_2}

Let $e \in \bK$ be a projection of rank one whose image is the one-dimensional trivial $G$-representation. Define the collection of maps
\[\theta^H_m\colon \eqAut{H}{(D \otimes \bK)^{\otimes m}} \rightarrow (\Omega^m KU^{D}_m)^H\]
by $\theta^H_m(\alpha)=(f \mapsto \hat{\eta}_m(f) \otimes \alpha((1_D \otimes e)^{\otimes m}))$, for all $m \in \N$ and $H \le G$. Since $\cG_D^{\bullet}$ has compatible inverses (see Lemma \ref{GD_stable}), these maps factor over a morphism
\[\theta: \cG_D^{\bullet} \rightarrow \text{GL}_1^{\bullet} KU^D\]
of commutative $\cI$-$\cO_G$-monoids. Corollary~\ref{cor:pikAut} suggests it might be worth investigating whether $\theta$ is a weak equivalence when $G=\Z/n\Z$. This turns out to be true, provided one restricts the target to the positive part with respect to the order structures on $\pi_0^H(KU^D) \cong K_0^H(D)$.

Assuming that $\pi_0(R)$ has an order structure we define $GL_1^\bullet(R)_+$ by a pullback diagram analogous to the one in \eqref{eqn:GL1_diag} with $GL_1(\pi_0(R))$ replaced by $GL_1(\pi_0(R))_+$. Let $gl_1(R)_+ = \underline{E}H_\cI GL_1^\bullet(R)_+$ be the associated naive $G$-spectrum and write $aut(D \otimes \bK) = \underline{E}H_{\cI}\cG_D^{\bullet}$.

\begin{theorem}\label{thm:main_thm}
There is a map of $\Z/n\Z$-spectra
\[ aut(D \otimes \bK) \rightarrow gl_{1}(KU^{D})\]
which is an isomorphism on all higher equivariant homotopy groups $\pi^{H}_m$ with $m >0$, and the inclusion $\text{GL}_1(K_{0}^{H}(D))_{+} \hookrightarrow \text{GL}_1(K_{0}^{H}(D))$ on $\pi_0^{H}$, for $H \le \Z/n\Z$. In particular, we have an equivalence of naive $\Z/n\Z$-spectra
\[
aut(D \otimes \bK) \simeq gl_{1}(KU^{D})_+\ .
\]
\end{theorem}

\begin{proof}
The morphism $\theta$ defined above induces the desired map of spectra. By Lemma \ref{deloop}, it suffices to check that the maps at level one
\[\theta_1^H: \eqAut{H}{D \otimes \bK} \rightarrow (\Omega KU_{1}^{D})^{H}\]
enjoy the prescribed properties, for every $H \le G$.
Fix $H \le G$. Observe that $\theta_1^H$ factors through
\[\eqAut{H}{D \otimes \bK} \xrightarrow{\Phi} \Proj{}{(D \otimes \bK)^H} \xrightarrow{\Psi} (\Omega KU_{1}^{D})^{H}\ ,\]
where $\Phi(\alpha)=\alpha(1_D \otimes e)$ and $\Psi(p)(f)=\hat{\eta}_{1}(f) \otimes p$. Under the identification (compare the proof of Theorem \ref{KU})
\[\pi_0((\Omega KU^D_1)^H) \cong K_0^H(D)\]
the class of $\Psi(p)$ corresponds to the equivariant $K$-theory class of $p$. Thus Lemma~\ref{units+} identifies the map induced by $\theta_1^H$ on components with the required inclusion of positive units.
For the higher homotopy groups we may restrict $\Phi$ to the homotopy equivalence (see Lemma~\ref{lem:hom_equivalences}~(\ref{it:AutProj}))
\[
    \Phi \colon \eqAutId{H}{D \otimes \bK} \to \Proj{1_D \otimes e}{(D \otimes \bK)^H}
\]
and $\Psi$ accordingly. Therefore the claim reduces to showing that
\[\pi_m(\text{Proj}((D \otimes \mathbb{K})^H), 1_D \otimes e) \xrightarrow{\Psi_*} \pi_m((\Omega KU^{D}_1)^H,\Psi(1_D \otimes e))\]
is an isomorphism for all $m \ge 1$.
Consider the following square:
\[\begin{tikzcd}
	{\pi_m(\text{Proj}((D \otimes \mathbb{K})^H), 1_D \otimes e)} & {\pi_m((\Omega KU^{D}_1)^H,\Psi(1_D \otimes e))} \\
	{K_0(C_0(S^m,*) \otimes D^H)} & {K_0(C_0(S^m,*) \otimes D^H)}
	\arrow["{\Psi_*}", from=1-1, to=1-2]
	\arrow["\beta"', from=1-1, to=2-1]
	\arrow["\cong", from=1-2, to=2-2]
	\arrow["{\text{id}}"', from=2-1, to=2-2]
\end{tikzcd}\]
Under the identification
\[
K_0(C_0(S^m,*)\otimes D^H)
\cong
K_0(C_0(S^m,*)\otimes(D\otimes\bK)^H)
\]
induced by the corner inclusion $a\mapsto a\otimes e$ (see Lemma~\ref{kt_fixed} and the discussion preceding Lemma~\ref{units+}), we have
\[
\beta([\gamma])=[p_\gamma]-
[1_{C(S^m)}\otimes1_D\otimes e]\ ,
\]
where $p_\gamma\in C(S^m)\otimes(D\otimes\bK)^H$
is the projection induced by a based map
$\gamma:S^m\to
\Proj{1_D\otimes e}{(D\otimes\bK)^H}$.
Since $p_\gamma(*)=1_D\otimes e$, this difference belongs to the kernel of evaluation at the basepoint, canonically identified with the corresponding reduced
$K$-theory group. 

Applying Lemma~\ref{lem:htpy_sets} and Bott periodicity as in the proof of
Theorem~\ref{KU}, a based map
$\varphi:S^m\to(\Omega KU_1^D)^H$ determines a class
\[
x_\varphi\in
K_0(C(S^m)\otimes(D\otimes\bK)^H)
\]
whose evaluation at the basepoint is $[1_D\otimes e]$.
The right hand vertical arrow sends $[\varphi]$ to
\[
x_\varphi-[1_{C(S^m)}\otimes1_D\otimes e]\ ,
\]
viewed in $K_0(C_0(S^m,*)\otimes D^H)$ via
the identification above.
The difference has zero evaluation at the basepoint, hence defines a reduced $K$-theory class, and adding back the constant basepoint class gives the inverse.
For $\varphi=\Psi\circ\gamma$, we have
$x_\varphi=[p_\gamma]$, so the right hand vertical arrow
yields $\beta([\gamma])$.
Therefore the square commutes.

The last step is to prove that $\beta$ is an isomorphism as well. It will follow that $\Psi$ induces isomorphisms on all higher homotopy groups, hence so does~$\theta$. For this purpose, let us consider the following diagram:
\[\begin{tikzcd}
	{\pi_{m-1}(U(D^H))} & {\pi_m(\text{Proj}_{1_D \otimes e}((D \otimes \mathbb{K})^H))} \\
	{K_1(C_0(S^{m-1},*) \otimes D^H)} & {K_0(C_0(S^m,*) \otimes D^H)}
	\arrow["\beta", from=1-2, to=2-2]
	\arrow["\cong"', from=1-1, to=2-1]
	\arrow["\cong", from=1-1, to=1-2]
	\arrow["\cong"', from=2-1, to=2-2]
\end{tikzcd}\]
Here the top horizontal arrow is the inverse of the isomorphism induced by the homotopy equivalence
\begin{equation} \label{eqn:Proj_equiv_BU}
\Proj{1_D\otimes e}{(D\otimes\bK)^H}
\simeq BU(D^H)
\end{equation}
of Lemma~\ref{lem:hom_equivalences}~(\ref{it:ProjBU}), composed with the clutching isomorphism
\[\pi_{m-1}(U(D^H))\cong\pi_m(BU(D^H))\ .\]

The space $BU(D^H)$ classifies right Hilbert $D^H$-module bundles with fibre~$D^H$. The pullback of the universal right Hilbert $D^H$-module bundle with respect to the homotopy equivalence \eqref{eqn:Proj_equiv_BU} is isomorphic to
\[\mathcal H_{\mathrm{univ}}=\left\{(p,\xi)\in
\Proj{1_D\otimes e}{(D\otimes\bK)^H} \times (D\otimes\bK)^H(1_D\otimes e):\ p\xi=\xi \right\},\]
where the right action is $\xi\cdot a=\xi(a\otimes e)$.
%Indeed, the identification with the associated bundle is given by
%\[ [u,a]\longmapsto
%\bigl(u(1_D\otimes e)u^*,\,u(a\otimes e)\bigr)\ ,\]
%where $u\in U(M((D\otimes\bK)^H))$ represents a point of the principal bundle and $a\in D^H$. 
%The homotopy equivalence \eqref{eqn:Proj_equiv_BU} is the classifying map of this associated bundle. 
For every based map
$f:S^m\to\Proj{1_D\otimes e}{(D\otimes\bK)^H}$ we have
\[\beta([f]) = [f^*\mathcal H_{\mathrm{univ}}]-[\mathcal H_{D^H}]\ ,\]
under the corner identification used in the definition of $\beta$. By the clutching description of the top horizontal arrow, the image of $[\gamma]$ is represented by a map $f_{\gamma}$ such that $f_{\gamma}^*\mathcal H_{\mathrm{univ}} \cong \mathcal H_\gamma =(S^{m-1} \times I \times D^H) \,/ \sim\,$, where $\sim$ is generated by
\[(z,0,a)\sim(z',0,a), \qquad (z,1,a)\sim (z',1,\gamma(z')^{-1}\gamma(z)a) \,\]
for $z,z'\in S^{m-1}$ and $a\in D^H$. 
Composition with $\beta$ therefore sends $[\gamma]$ to $[\mathcal H_\gamma]-[\mathcal H_{D^H}]$, where $\mathcal H_{D^H}=S^m\times D^H$ is the trivial bundle.

The left hand vertical arrow is the reduced version of the isomorphism described in Remark~\ref{K1}, and the bottom horizontal arrow is the suspension isomorphism in $K$-theory. The description of this map in
\cite[Ex.~2.2.25]{paper:Thesis} shows that this composite also maps $[\gamma]$ to $[\mathcal{H}_{\gamma}]-[\mathcal{H}_{D^H}]$. Therefore the diagram commutes, implying that the map $\beta$ is an isomorphism, thus concluding the proof.
\end{proof}

Our final result combines Lemma~\ref{lem:comp_vs_tensor} with Theorem~\ref{thm:main_thm}.

\begin{corollary} \label{cor:bdl_classification}
    $\Aut{D \otimes \bK}$ is an equivariant infinite loop space with associated $\Z/n\Z$-equivariant cohomology theory $E_D^*(X)=gl_1(KU^D)_+^*(X)$, and
    \[E^0_D(X) \cong [X,\Aut{D \otimes \bK}]_{\Z/n\Z} \quad \text{and} \quad E^1_D(X) \cong [X,B\Aut{D \otimes \bK}]_{\Z/n\Z}\ .\]
    In particular, isomorphism classes of locally trivial $(\Z/n\Z, D \otimes \bK)$-bundles over the finite $\Z/n\Z$-$CW$-complex $X$ form a group with respect to the fibrewise tensor product that is isomorphic to $E^1_D(X)$.
\end{corollary}

\appendix
\section{The equivariant homotopy type of $\Aut{D \otimes \bK}$}
Let $G = \Z/n\Z$, denote by $D = \Endo{V}^{\otimes \infty}$ the UHF-algebra associated to a finite-dimensional unitary $G$-representation $V$ and let $\bK = \bK(\ell^2(G) \otimes H_0)$ for an infinite-dimensional separable Hilbert space $H_0$. Let $\Proj{}{D \otimes \bK}$ be the space of projections in $D \otimes \bK$ equipped with the subspace topology and denote by $\Aut{D \otimes \bK}$ the group of $*$-automorphisms with the point-norm topology considered as a $G$-space with the induced action. Throughout we write $q_0 = 1_D \otimes e$, where $e \in \bK$ is a rank one projection whose image is the trivial $G$-representation (in particular, $q_0$ is $G$-fixed).

In this appendix we will show that the automorphism group $\Aut{D \otimes \bK}$ is equivariantly well-pointed in the sense that the inclusion of the identity is an equivariant cofibration. Apart from this, we will also prove a few technical results that might be of independent interest: In particular, we will see that $\Aut{D \otimes \bK}$ is $G$-homotopy equivalent to a $G$-CW-complex. 

We use the following terminology, see \cite[\S I.6]{book:tomDieck} or \cite[Def.~1.7]{paper:Lueck}. A \emph{family}~$\Fam$ of subgroups of a finite group $H$ is a set of subgroups closed under conjugation and under passing to subgroups. A \emph{universal $\Fam$-space} $\EFam{\Fam}{H}$ is an $H$-CW-complex whose isotropy groups lie in $\Fam$ and which is terminal in the $H$-homotopy category of $H$-CW-complexes with isotropy in $\Fam$; equivalently, $(\EFam{\Fam}{H})^K$ is contractible for $K \in \Fam$ and empty otherwise. It exists and is unique up to $H$-homotopy equivalence \cite[Thm.~1.9]{paper:Lueck}. An $H$-space $Y$ is \emph{$\Fam$-numerable} if it admits an open cover by $H$-subspaces $U_i$ together with $H$-maps $U_i \to H/K_i$ with $K_i \in \Fam$ and a locally finite subordinate partition of unity by $H$-invariant functions \cite[Def.~2.1]{paper:Lueck}. 

%As we will see in the next lemma with respect to finite group actions the space of projections $\Proj{}{A}$ of a $G$-$C^*$-algebra also has the $G$-homotopy type of a $G$-CW-complex.
\begin{lemma} \label{lem:Proj_GCW}
    Let $G$ be a finite group and let $A$ be a $G$-$C^*$-algebra. The subspace $\Proj{}{A} \subset A$ has the $G$-homotopy type of a $G$-CW-complex. Let $I \subseteq \pi_0(\Proj{}{A})$ be a $G$-invariant subset and let $\Proj{I}{A}$ be the pullback of $\Proj{}{A}$ over $I$. Then the space
    \begin{equation} \label{eqn:Proj_I_def}
        \Proj{I}{A} = \coprod_{[q] \in I} \Proj{q}{A}
    \end{equation}
    has the $G$-homotopy type of a $G$-CW-complex as well.
\end{lemma}

\begin{proof}
    By \cite{paper:Antonyan} the real Banach space $A_{sa}$ of self-adjoint elements in the algebra $A$ is a $G$-ANE ($G$-equivariant absolute neighbourhood extensor). Therefore the $G$-invariant open subset 
    \[
        A_{sa}^0 = \{ x \in A_{sa} \ |\ \tfrac{1}{2} \notin \sigma(x) \} \subset A_{sa}
    \]
    is a $G$-ANE as well \cite[Prop.~6.1]{paper:Murayama}. By \cite[Prop.~6.4]{paper:Murayama} it is then also a $G$-ANR ($G$-equivariant absolute neighbourhood retract) and \cite[Thm.~13.3]{paper:Murayama} implies that it has the $G$-homotopy type of a $G$-CW-complex. Note that 
    \[
        \Proj{}{A} \subset A_{sa}^0\ .
    \]
    Applying continuous functional calculus with the map that is $0$ on $(-\infty, \tfrac{1}{2})$ and $1$ on $(\tfrac{1}{2}, \infty)$ to the elements of $A_{sa}^0$ we obtain a $G$-equivariant map 
    \[
        \rho \colon A_{sa}^0 \to \Proj{}{A}\ ,
    \]
    that restricts to the identity on $\Proj{}{A}$. In particular, $\Proj{}{A}$ is dominated by a space with the $G$-homotopy type of a $G$-CW-complex. Therefore \cite[Lem.~4.7]{paper:Waner} implies that it has the $G$-homotopy type of a $G$-CW-complex itself\footnote{The lemma in \cite{paper:Waner} is stated for spaces that are dominated by $G$-CW-complexes, but the argument actually only requires the dominating space to have the correct $G$-homotopy type.}. The second statement is a consequence of the first.
\end{proof}

In the case of finite group actions evaluating an automorphism at a fixed projection gives an equivariant Hurewicz fibration as we will see below.

\begin{lemma} \label{lem:Hurewicz}
Let $G$ be a finite group and let $A$ be a $G$-$C^*$-algebra. Let $q_0 \in \Proj{}{A}$ be a projection, which is fixed by $G$.
Let 
\[
    \Xi \colon \Aut{A} \to \Proj{I}{A} \quad, \quad \beta \mapsto \beta(q_0)
\]
where $I \subseteq \pi_0(\Proj{}{A})$ is the image of $\Xi_*$ on $\pi_0$ and $\Proj{I}{A}$ is defined as in \eqref{eqn:Proj_I_def}. Then $\Xi$ is a Hurewicz $G$-fibration (where $G$ acts on the domain by conjugation). 
\end{lemma}

\begin{proof}
We will show that $\Xi \colon \Aut{A} \to \Proj{I}{A}$ is a locally trivial $G$-fibration in the sense of \cite[Sec.~5]{paper:BoothHoskins}. Since the codomain is a metric space, hence paracompact, the result follows from \cite[Lem.~5.1]{paper:BoothHoskins}, \cite[Lem.~1.12]{paper:LashofEquivBundles} and \cite[Lem.~3.2.4]{paper:WanerClassFib}. Let $p_0 \in \Proj{I}{A}$ and $H = \text{stab}_G(p_0)\le G$ be the stabiliser of $p_0$. For two projections $s, t \in \Proj{I}{A}$ we define 
\[
    \Autf{s}{t}{A} = \{ \beta \in \Aut{A}\ |\ \beta(s) = t \} \subseteq \Aut{A}\ .
\]
The group $H$ acts on $\Autf{q_0}{p_0}{A}$ by conjugation and 
\[
    G \times_H \Autf{q_0}{p_0}{A} \cong \Autf{q_0}{Gp_0}{A} \to G/H
\]
with $\Autf{q_0}{Gp_0}{A} = \{\beta \in \Aut{A}\ |\ \beta(q_0) \in G \cdot p_0\}$ is $G$-equivariant and a $G$-model trivial fibration \cite[Sec.~5]{paper:BoothHoskins}. Define
\[    
    U_0 = \{ q \in \Proj{I}{A}\ | \ \lVert q-p_0 \rVert < 1\} \ .
\]
The subspace $U_0$ is $H$-invariant. As a metric space $\Proj{I}{A}$ is completely regular. By \cite[Thm.~1.7.19]{paper:Palais} there exists a slice $V \subseteq \Proj{I}{A}$ with $p_0 \in V$ and $V \cap U_0$ provides another slice containing $p_0$. Let 
\[
    U = G \cdot (V \cap U_0) \cong G \times_H (V \cap U_0)\ .
\]
By \cite[Prop.~II.3.3.4]{book:BlackadarOpAlg} there is a continuous map
\(
    u_{p_0} \colon U_0 \to U(M(A))
\)
with the properties 
\begin{enumerate}[i)]
    \item \label{it:transport} $u_{p_0}(q)\,p_0\,u_{p_0}(q)^* = q$,
    \item \label{it:equiv} $\alpha_h(u_{p_0}(q)) = u_{p_0}(\alpha_h(q))$ for all $h \in H$.
\end{enumerate}
Now consider the continuous map
\[
   \kappa \colon G \times_H ((V \cap U_0) \times \Autf{q_0}{p_0}{A}) \to \Xi^{-1}(U)
\]
given by $\kappa([g, (q,\beta)]) = \alpha_g \circ \Ad_{u_{p_0}(q)} \circ \beta \circ \alpha_g^{-1}$. Since
\[
    \alpha_{gh} \circ \Ad_{u_{p_0}(\alpha_h^{-1}(q))} \circ\ (\alpha_h^{-1} \circ \beta \circ \alpha_h) \circ \alpha_{gh}^{-1} = \alpha_g \circ \Ad_{u_{p_0}(q)} \circ \beta \circ \alpha_g^{-1}
\]
this is well-defined. It is also equivariant with respect to the left action of~$G$ on the domain and the conjugation action on the codomain. To define the inverse map, choose coset representatives $S = \{g_0 = e, g_1, \dots, g_n\} \subset G$ for~$G/H$. Observe that for $\beta \in \Xi^{-1}(U)$ the projection $\beta(q_0) \in G \cdot (V \cap U_0)$ can be written as $\alpha_{g_\beta}(q_\beta)$ for a unique $g_\beta \in S$ and a projection~$q_\beta$ with $\lVert q_\beta - p_0 \rVert < 1$. Define
\[
    \theta \colon \Xi^{-1}(U) \to G \times_H ((V \cap U_0) \times \Autf{q_0}{p_0}{A})
\]
by $\theta(\beta) = [g_\beta, (q_\beta, \Ad_{u_{p_0}(q_\beta)^*} \circ\ 
\alpha_{g_\beta}^{-1} \circ \beta \circ \alpha_{g_\beta})]$. A straightforward computation shows that $\theta \circ \kappa = \id{}$ and $\kappa \circ \theta = \id{}$, which implies that $\theta$ is $G$-equivariant as well. Hence, $\kappa$ is a local trivialisation.
\end{proof}

\begin{remark} \label{rem:torsor}
    The space $\Autf{q_0}{p_0}{A}$ is a right $\AutSt{q_0}{A}$-torsor. The left action of $G$ on $\Aut{A}$ commutes with the right action by $\AutSt{q_0}{A}$ and $\Xi$ is equivariant with respect to those two actions. 
    %Replacing conjugation actions by left actions everywhere in the proof one can show that $\Xi \colon \Aut{A} \to \Proj{I}{A}$ is a $G$-equivariant principal bundle with structure group $\AutSt{q_0}{A}$ in analogy to \cite[Lem.~2.8]{paper:DP-Dixmier}. 
    We will use the torsor structure in the following two forms. First, if $\beta_0 \in \Autf{q_0}{p_0}{A}$, then
    \[
        \AutSt{q_0}{A} \to \Autf{q_0}{p_0}{A} \quad , \quad \beta \mapsto \beta_0 \circ \beta
    \]
    is a homeomorphism, and it is $K$-equivariant for a subgroup $K \le \text{stab}_G(p_0)$ acting by conjugation on both sides as soon as $\beta_0$ is $K$-equivariant. Secondly, the map
    \[
        \Autf{q_0}{p_0}{A} \times \Autf{q_0}{p_0}{A} \to \AutSt{q_0}{A} \quad , \quad (\beta, \beta') \mapsto \beta^{-1} \circ \beta'
    \]
    is continuous and satisfies $\alpha_h \circ (\beta^{-1} \circ \beta') \circ \alpha_h^{-1} = (\alpha_h \beta \alpha_h^{-1})^{-1} \circ (\alpha_h \beta' \alpha_h^{-1})$ for $h \in \text{stab}_G(p_0)$.
\end{remark}

We can now identify the fibres of $\Xi$.

\begin{lemma} \label{lem:fibres}
     Let $D$, $\Xi$, $q_0$ and $I$ be as in Lemma~\ref{lem:Hurewicz} and let $A = D \otimes \bK$, let $p \in \Proj{I}{A}$ with stabiliser $H = \text{stab}_G(p) \le G$. Set
    \[
        \Fam_p = \left\{ K \le H \ |\ (\Autf{q_0}{p}{A})^K \neq \emptyset \right\}\ .
    \]
    \begin{enumerate}
        \item \label{it:fib_family} $\Fam_p$ is a family of subgroups of $H$, all isotropy groups of the $H$-space $\Autf{q_0}{p}{A}$ belong to $\Fam_p$, and $\Autf{q_0}{p}{A}$ is $K$-equivariantly contractible for every $K \in \Fam_p$.
        \item \label{it:fib_universal} $\Autf{q_0}{p}{A}$ is $H$-homotopy equivalent to $\EFam{\Fam_p}{H}$. Consequently
        \[
            \Xi^{-1}(G \cdot p) \cong G \times_H \Autf{q_0}{p}{A} \simeq_G G \times_H \EFam{\Fam_p}{H}\ ,
        \]
        which is a $G$-CW-complex.
        %\item \label{it:fib_clopen} For every subgroup $K \le G$ the set $\{ q \in \Proj{I}{A}^K \ |\ F_q^K \neq \emptyset \}$ is open and closed in $\Proj{I}{A}^K$.
    \end{enumerate}
\end{lemma}

\begin{proof}
    (\ref{it:fib_family}) Since $p \in \Proj{I}{A}$ the fibre $\Autf{q_0}{p}{A}$ is non-empty. Note that 
    \[
        (\Autf{q_0}{p}{A})^K = \left\{ \beta \in \eqAut{K}{A} \ |\ \beta(q_0) = p \right\}\ .
    \]
    Hence, if $K' \le K$, then $(\Autf{q_0}{p}{A})^K \subseteq (\Autf{q_0}{p}{A})^{K'}$. Conjugation by $h \in H$ maps $(\Autf{q_0}{p}{A})^K$ onto $(\Autf{q_0}{p}{A})^{hKh^{-1}}$. Therefore $\Fam_p$ is closed under subgroups and conjugation. If $\beta \in \Autf{q_0}{p}{A}$ has iso\-tropy~$K$ then $K \in \Fam_p$ by definition. For $K \in \Fam_p$ let $\beta_0 \in (\Autf{q_0}{p}{A})^K$. The map $\beta \mapsto \beta_0 \circ \beta$ provides a $K$-equivariant homeomorphism $\AutSt{q_0}{A} \to \Autf{q_0}{p}{A}$ by Remark~\ref{rem:torsor}, so $\Autf{q_0}{p}{A}$ is $K$-equivariantly contractible by Lemma~\ref{lem:hom_equivalences}~\eqref{it:Aut_contractible}.

    (\ref{it:fib_universal}) We will first use \cite[Thm.~2.5~(ii)]{paper:Lueck} to see that $\Autf{q_0}{p}{A}$ is a terminal object in the $H$-homotopy category of $\Fam_p$-numerable $H$-spaces.
    
    To see that $\Autf{q_0}{p}{A}$ is $\Fam_p$-numerable observe that it is separable and metrisable, since it is homeomorphic to $\AutSt{q_0}{A}$. Let $\beta \in \Autf{q_0}{p}{A}$ with stabiliser $K$. Since $H$ is finite and $\Autf{q_0}{p}{A}$ is Hausdorff, there exists an open neighbourhood $U \ni \beta$ with $KU = U$ and $hU \cap U = \emptyset$ for $h \in H \setminus K$. We have $H \cdot U = H \times_K U$ and the projection $H\cdot U \to H/K$ is $H$-equivariant. These sets cover $\Autf{q_0}{p}{A}$. Passing to a countable locally finite refinement and averaging a subordinate partition of unity over the finite group $H$ produces the required $H$-invariant partition of unity.
    
    %Let $\beta \in F_p$ with isotropy $K = H_\beta \in \Fam_p$. Since $H$ is finite and $F_p$ is Hausdorff, there is an open neighbourhood $U$ of $\beta$ with $KU = U$ and $hU \cap U = \emptyset$ for $h \in H \setminus K$; then $H \cdot U \cong H \times_K U$ and the projection gives an $H$-map $H \cdot U \to H/K$. These sets cover $F_p$; 

    %\emph{The two projections $\mathrm{pr}_1, \mathrm{pr}_2 \colon F_p \times F_p \to F_p$ are $H$-homotopic.} 
    Let $\Theta$ be as the homotopy from Lemma~\ref{lem:hom_equivalences}~\eqref{it:Aut_contractible} and define
    \begin{gather*}
        \Lambda \colon \Autf{q_0}{p}{A} \times \Autf{q_0}{p}{A} \times [0,1] \to \Autf{q_0}{p}{A} \\  \Lambda(\beta, \beta', t) = \beta \circ \Theta(\beta^{-1} \circ \beta', t)\ .
    \end{gather*}
    We have $\Lambda(\beta,\beta',0) = \beta \circ \beta^{-1} \circ \beta' = \beta'$ and $\Lambda(\beta,\beta',1) = \beta$. For $h \in H$ the two identities in Remark~\ref{rem:torsor} together with the $G$-equivariance of $\Theta$ give
    \[
        \alpha_h \circ \Lambda(\beta,\beta',t) \circ \alpha_h^{-1} = \Lambda(\alpha_h \beta \alpha_h^{-1}, \alpha_h \beta' \alpha_h^{-1}, t)\ ,
    \]
    so $\Lambda$ is an $H$-homotopy from $\text{pr}_1 \colon \Autf{q_0}{p}{A} \times \Autf{q_0}{p}{A} \to \Autf{q_0}{p}{A}$ to $\text{pr}_2$.

    Therefore by \cite[Thm.~2.5~(ii)]{paper:Lueck} the space $\Autf{q_0}{p}{A}$ is a model for $J_{\Fam_p}(H)$ (see \cite[Thm.~2.5]{paper:Lueck}). The canonical $H$-map $\EFam{\Fam_p}{H} \to J_{\Fam_p}(H)$ is an $H$-homotopy equivalence by \cite[Thm.~3.7~(ii)]{paper:Lueck}, see also \cite[Thm.~I.6.6]{book:tomDieck}. This proves $\Autf{q_0}{p}{A} \simeq_H \EFam{\Fam_p}{H}$. The description of $\Xi^{-1}(G\cdot p)$ is contained in the proof of Lemma~\ref{lem:Hurewicz}, and induction along $H \le G$ turns the $H$-CW-complex $\EFam{\Fam_p}{H}$ into a $G$-CW-complex.
    %(\ref{it:fib_clopen}) Let $q, q' \in \Proj{I}{A}^K$ with $\lVert q' - q \rVert < 1$. By property~\ref{it:equiv}) of the map $u_{q}$ in the proof of Lemma~\ref{lem:Hurewicz}, applied with $H$ replaced by $K$, the unitary $u_{q}(q') \in U(M(A))$ is $K$-invariant, so $\Ad_{u_{q}(q')} \in \eqAut{K}{A}$ and $\Ad_{u_{q}(q')}(q) = q'$. Composing with an element of $F_{q}^K$ therefore shows $F_{q'}^K \neq \emptyset$, and composing with an element of $F_{q'}^K$ shows the converse. Hence both the set and its complement are open.
\end{proof}

\begin{theorem} \label{lem:GCW-complex}
Let $G$, $V$, $D$ and $\bK$ be as above. The topological group $\Aut{D \otimes \bK}$ equipped with the conjugation action of $G$ has the $G$-homotopy type of a $G$-CW-complex.
\end{theorem}

\begin{proof}
Let $A = D \otimes \bK$. By Lemma~\ref{lem:Hurewicz} the evaluation map $\Xi \colon \Aut{A} \to \Proj{I}{A}$ is a Hurewicz $G$-fibration, and by Lemma~\ref{lem:Proj_GCW} its base space has the $G$-homotopy type of a $G$-CW-complex. Since $G$ is finite and $\Proj{I}{A}$ is metrisable, the action on the base is proper \cite[Cor.~3.1.7]{book:MaySigurdsson}. For $p \in \Proj{I}{A}$ with stabiliser $H = \text{stab}_G(p)$, the fibre $\Xi^{-1}(p)$ has the $H$-homotopy type of the $H$-CW-complex $\EFam{\Fam_p}{H}$ by Lemma~\ref{lem:fibres}~(\ref{it:fib_universal}). Hence $\Aut{A}$ has the $G$-homotopy type of a $G$-CW-complex by \cite[Thm.~3.5.2]{book:MaySigurdsson}.
\end{proof}

\begin{lemma} \label{lem:G-cofib}
Let $G$, $V$, $D$ and $\bK$ be as in the first paragraph of the appendix. The topological group $\Aut{D \otimes \bK}$ equipped with the conjugation action of $G$ is equivariantly well-pointed in the sense that the inclusion $\{\id{}\} \to \Aut{D \otimes \bK}$ is a $G$-cofibration.
\end{lemma}

\begin{proof}
    %By a straightforward $G$-equivariant generalisation of \cite[Thm.~12]{paper:Strom} 
    By \cite[Lem.~2.3]{paper:Grant} it suffices to construct a $G$-invariant continuous map $v \colon \Aut{D \otimes \bK} \to [0,1]$ such that 
    \begin{enumerate}
        \item $v^{-1}(0) = \{\id{}\}$,
        \item $U_v = v^{-1}([0,1)) = \{ \beta \in \Aut{D \otimes \bK} \ |\ v(\beta) < 1\}$ deformation retracts $G$-equivariantly to $\id{}$.
    \end{enumerate}
    Pick an arbitrary countable dense subset $S = \{ a_k \in D\otimes \bK \ |\ k \in \N\}$. The point-norm topology on $\Aut{D \otimes \bK}$ is the metric topology associated to the metric
   \[
        d(\beta_1, \beta_2) = \sum_{k \in \N} \frac{1}{2^k} \sum_{g \in G} \frac{\lVert \beta_1(\alpha_g(a_k)) - \beta_2(\alpha_g(a_k)) \rVert}{1+ \lVert \alpha_g(a_k) \rVert},
    \]
    which is $G$-invariant with respect to the conjugation action. Let $\Xi$ be as in Lemma~\ref{lem:Hurewicz}. We may now proceed as in \cite[Prop.~2.26]{paper:DP-Dixmier}. Let
    \[
        v(\beta) = \max\{\min\{d(\beta, \id{}), \tfrac{1}{2}\}, \min\{1, 2\lVert \beta(q_0) - q_0 \rVert\} \}\ ,
    \]
    which is $G$-invariant, continuous and satisfies $v^{-1}(0) = \{\id{}\}$. Note that
    \[
        U_v = \Xi^{-1}(W) \quad \text{with} \quad W = \{ q \in \Proj{q_0}{D \otimes \bK} \ |\ \lVert q - q_0 \rVert < \tfrac{1}{2} \}\ .
    \]
    As in the proof of Lemma~\ref{lem:Hurewicz} the map
    \[
        \kappa \colon W \times \AutSt{q_0}{D\otimes \bK} \to \Xi^{-1}(W)
    \]
    with $\kappa(q,\beta) = \Ad_{u_{q_0}(q)} \circ \beta$ is a $G$-equivariant homeomorphism. Hence, it suffices to prove that the domain equivariantly deformation retracts to $(q_0, \id{})$. Let $\chi$ be the characteristic function of $(\tfrac{1}{2},1]$. The $G$-equivariance of functional calculus and the choice of $W$ ensure that $h(q,t) = \chi((1-t)q + tq_0)$ is a well-defined equivariant retraction of $W$ to $q_0$. By Lemma~\ref{lem:hom_equivalences}~\eqref{it:Aut_contractible} the group $\AutSt{q_0}{D \otimes \bK}$ equivariantly deformation retracts to $\id{}$. Combining both retractions gives the result.
\end{proof}
{
%\small
\bibliographystyle{plain}
\bibliography{main}
}
\end{document}